\documentclass[a4paper, 10pt, reqno, final]{article}
\usepackage{caption}
\usepackage{float}
\usepackage[T1]{fontenc}
\usepackage[utf8]{inputenc}
\usepackage[english]{babel}
\usepackage{amsmath}
\usepackage{amssymb}
\usepackage{asymptote}
\usepackage{amsthm}
	\theoremstyle{plain}
		\newtheorem{mainthm}{\textsc{Theorem}}

		\newtheorem{thm}{Theorem}[section]

		\newtheorem{cor}[thm]{Corollary}	 
		\newtheorem{maincor}{\textsc{Corollary}}

		\newtheorem{lem}[thm]{Lemma}		

		\newtheorem{prop}[thm]{Proposition}

	\theoremstyle{definition}
		\newtheorem{defn}[thm]{Definition}

	\theoremstyle{remark}
		\newtheorem{rem}[thm]{Remark}

\numberwithin{equation}{section}
\usepackage{mathtools}
\usepackage{mathrsfs}
\usepackage{eucal}
\usepackage{braket}
\usepackage{a4wide}
\usepackage{booktabs}		
\usepackage{multirow}		
\usepackage{caption}		
\usepackage{rotating}
\usepackage{subfig}
\usepackage{varioref}		
\usepackage{footmisc}

\usepackage{graphicx}		
\usepackage{epsfig}
\usepackage{subfig}
\usepackage{enumerate}
\usepackage[colorlinks=true, linkcolor=black, citecolor=black, 
urlcolor=blue]{hyperref}		
\mathtoolsset{showonlyrefs}

\newcommand{\R}{\mathbb{R}}
\newcommand{\C}{\mathbb{C}}

\newcommand{\Z}{\mathbb{Z}}	
	
\renewcommand{\P}{\mathscr{P}}
\newcommand{\Lagr}{\Lambda}

\newcommand{\norm}[1]{\left\| #1 \right\|}		
\newcommand{\cfsa}{\mathscr{CF}^{sa}}
\newcommand{\im}{\mathrm{rge}\,}
\newcommand{\Mat}{\mathrm{Mat}\,}
\newcommand{\Sp}{\mathrm{Sp}\,}
\newcommand{\Lin}{\mathscr{L}}
\newcommand{\GL}{\mathrm{GL}}
\newcommand{\Id}{I}
\newcommand{\br}{~\\}

\DeclareMathOperator{\spfl}{sf}
\DeclareMathOperator{\sgn}{sgn}
\DeclareMathOperator{\iCLM}{\mu^{\scriptscriptstyle{\mathrm{CLM}}}}
\DeclareMathOperator{\irel}{i_{\textup{rel}}}
\DeclareMathOperator{\iRS}{\mu^{\textup{RS}}}
\DeclareMathOperator{\ispec}{i_{\textup{spec}}}
\DeclareMathOperator{\igeo}{\iota}
\DeclareMathOperator{\igeob}{\iota_b}
\renewcommand{\leq}{\leqslant}
\renewcommand{\geq}{\geqslant}
\renewcommand{\hat}{\widehat}
\renewcommand{\tilde}{\widetilde}
\renewcommand{\=}{\coloneqq}

\newcommand{\email}[1]{\href{mailto:#1}{\textsf{#1}}}

\title{Index theory for heteroclinic orbits\\ of\\  Hamiltonian systems}
\author{Xijun Hu\thanks{The author is partially supported by NSFC( No.11425105) 
and NCET.},  
Alessandro Portaluri, 
\thanks{The   
author is partially supported by the project ERC Advanced Grant 2013 
No.~339958 ``Complex Patterns for Strongly Interacting Dynamical Systems --- 
COMPAT” and by project  ``Semi-classical trace formulas and their application 
in physical chemistry". 
Ricerca locale 2015 No. Borr$\_$Rilo$\_$16$\_$01.} }
\date{\today}
\begin{document}
 \maketitle
 
% %  \nocite{*}
\begin{abstract}
Index theory revealed its outstanding role in the study of 
periodic orbits of  Hamiltonian systems and the  dynamical consequences of 
this theory are enormous. 
Although  the index theory in the periodic case is 
well-established,  very few  results are known in the case of 
homoclinic orbits of Hamiltonian systems. Moreover, to the authors' knowledge, 
no 
results have been yet proved in the case of heteroclinic 
and halfclinic (i.e. parametrised by a half-line) orbits.

Motivated by the importance played by these motions  in understanding several 
challenging  problems in Classical Mechanics,  we develop a new index theory and 
we prove at once 
a general spectral flow formula for  heteroclinic, homoclinic and halfclinic 
trajectories. Finally we show how this  index theory can be used  to 
recover all the (classical) existing  results on orbits parametrised by bounded 
intervals. 

\end{abstract}
\tableofcontents

% % % % % % % % % % % % % % % % % % % % % % % % % % % % % % % % % % % % % % % % 
% % % % % % % % % % % % % % % % % % % % % 
% % % 
% % % 
% % % 
% % % % % % % % % % % % % % % % % % % % % % % % % % % % % % % % % % % % % % % % 
% % % % % % % % % % % % % % % % % % % % % % 

\section*{Introduction}

Several central problems in Classical Mechanics involve unbounded trajectories 
of a phase flow or, more generally, of a one-parameter family of phase flows. 
Hamiltonian PDE's, e.g.  reaction-diffusion equations in one space dimension, 
such as fifth-order Kortweg-De-Vries (KdV), nonlinear 
Schr\"odinger (NLS) or longwave-shortwave resonance (LW-SW) equations, have the 
property that their steady 
part is a finite-dimensional Hamiltonian system. For these Hamiltonian systems, 
solitary wave solutions can be characterised 
as homoclinic and heteroclinic orbits, namely 
solutions parametrised by the whole real line 
and joining a saddle point to itself in  the former case, two different 
saddle points in the 
latter case. In this respect, the spectral problem 
associated with the linearisation around a 
given homoclinic motion leads to a one-parameter family of linear
Hamiltonian systems. 

Central configurations in the $N$-body problem with a general singular 
$\alpha$-homogeneous weak self-interacting potential (including the gravitation 
case) give rise 
to special asymptotic solutions (e.g. homographic as well as a class of 
colliding or 
parabolic motions) that represent an interesting 
class of motions parametrised by  {\em unbounded\/} intervals for which the 
index theory 
developed in  this paper could be successfully employed. (We refer the reader 
to 
\cite{HO16,BHPT17} and references therein). 
In the classical case of  orbits 
parametrised by   {\em bounded\/} intervals (for instance in the study of  
periodic orbits) 
spectral flow  formulas have been recently used in order  to tackle challenging 
linear 
stability problems. (We  refer the interested reader to \cite{HS09, HS10, 
HLS14, 
BJP14, BJP16} and references therein). 
Except these results in which an index theory for  homoclinic motions 
was developed, 
we only mention the paper \cite{CH07}, where the  authors assigned a 
geometrical index to any unbounded 
motion of a Hamiltonian system, and  \cite{Wat15} where a suitable 
spectral flow 
formula for a one-parameter 
family of Hamiltonian systems under homoclinic boundary conditions  was proved. 

Two are the main ingredients of the index theory developed in this paper. The  
first one is essentially based on a symplectic invariant known in literature as 
{\em Maslov index\/}, which, roughly speaking, 
counts algebraically  the intersections between a (continuous) curve of 
Lagrangian subspaces $\ell$  and 
a transversally oriented sub-variety (the singular Maslov cycle $\Sigma(L_0)$) 
of a fixed Lagrangian subspace $L_0$ in the Lagrangian Grassmannian manifold. 
\begin{figure}[ht]
 \centering
 \includegraphics[scale=0.2]{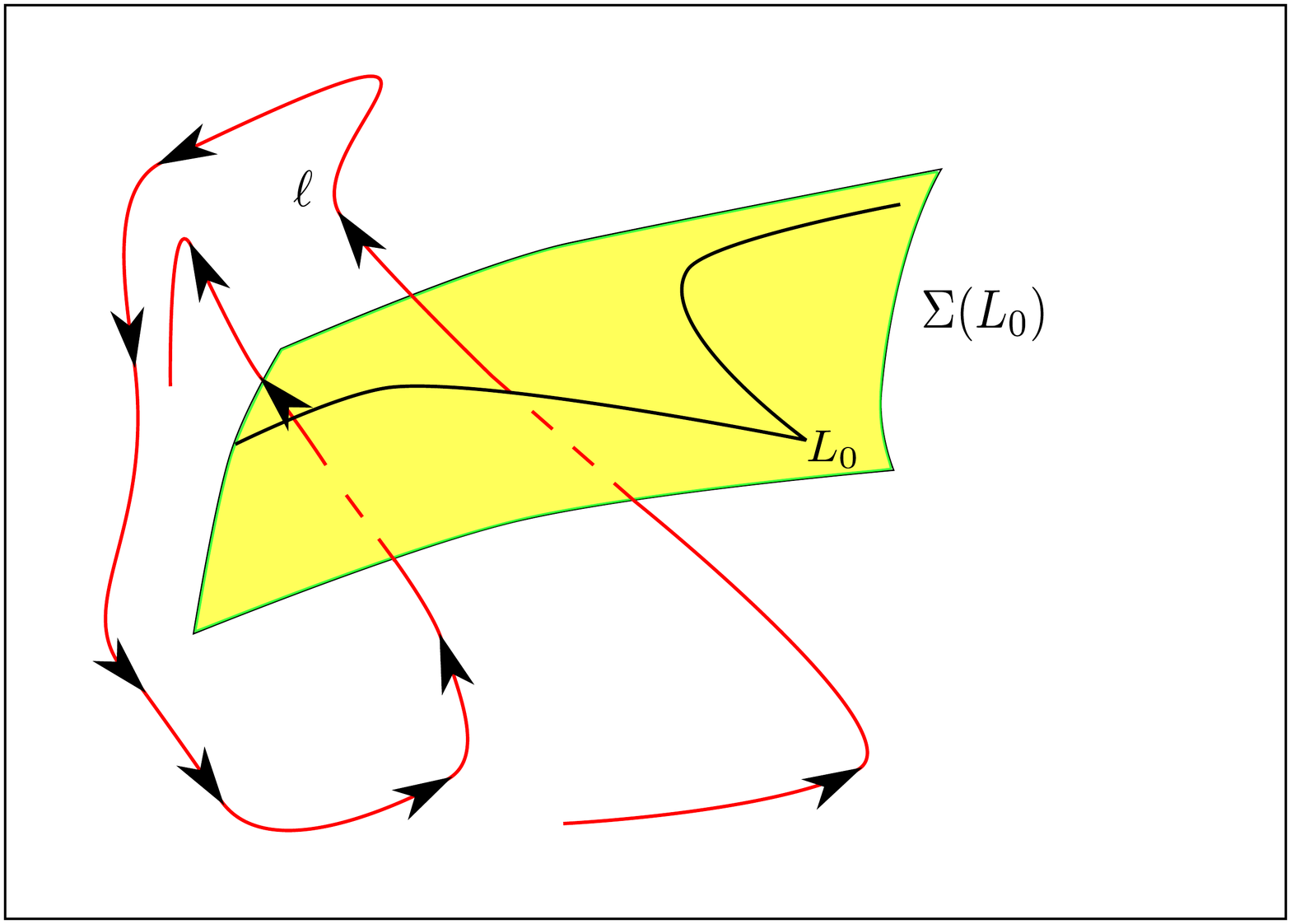}
\caption{$\Sigma(L_0)$ is the Maslov cycle (or train) with vertex $L_0$. The 
(red)  curve $\ell$ represents a 
smooth path  of Lagrangian subspaces. The black curve inside the cycle 
$\Sigma(L_0)$ represents the (singular)
higher codimensional stratum.}\label{fig:Maslov}
 \end{figure}
A strictly related notion that we essentially consider, is 
the Maslov index for (ordered) pairs of Lagrangian paths. (Cf. Section 
\ref{sec:spectral-flow-formula}, 
\cite{RS93, CLM94, ZL99} and reference therein). 

The second key ingredient is a well-known  topological invariant termed {\em 
spectral flow. \/}
The spectral flow is an integer-valued  homotopy invariant for paths of 
selfadjoint Fredholm 
operators that was introduced by 
Atiyah, Patodi and Singer in \cite{APS76}  in their investigation of index 
theory on manifolds with boundary. Roughly speaking, 
it counts the net  number of eigenvalues which pass through the zero in the 
positive direction when the parameter travels along the whole interval. 
Otherwise stated is the integer given by the number of 
negative eigenvalues that 
becomes positive minus the number of positive that become negative when the 
parameter runs along the whole interval. 
A related notion that we shall use is that of 
{\em $\varepsilon$-spectral flow\/} (cf. \cite{CLM94} and references therein),
denoted by $\varepsilon$-$\spfl$  which counts the 
number of eigenvalues $\lambda$ crossing the line $\lambda=\varepsilon >0$ with 
sign.  Due to its central importance in the (symplectic) 
Sturm theory (cf. \cite{Arn86}), in 
the study of conjugate and focal points along geodesics in semi-Riemannian 
manifolds (cf. \cite{GPP04, 
MPP05, MPP07, PPT04} and references therein), as well as  grading the Floer 
complex in  Floer homology (cf. \cite{APS08} and references therein),
only to mention very few of them, the literature on the subject is 
very broad. Among all of them we found particularly interesting and  
elegant at the same time the following papers that 
represent our basic references: \cite{CLM94,RS95,Phi96,FPR99,ZL99,BLP05}. 
Quoting the masterpiece of V.I. Arnol'd \cite{Arn67}, the idea behind 
the spectral flow formula can  be masterfully resumed 
by the following few lines
\begin{quote}
{\em [...] It turned out that there appeared in the asymptotic formulas certain 
integers, reflecting homological 
properties of curves on surfaces of the phase space and closely related to the 
Morse indices of the 
corresponding variational problems.\/}
\end{quote}
This  deep sentence essentially summarises the content of this 
paper in the case of  Hamiltonian systems and 
motions parametrised by (un)bounded intervals. So far, 
several constructions of the Maslov index and related indices (e.g the 
Maslov-type, 
the Conley-Zehnder, H\"ormander or four-fold, Kashiwara or  triple, Leray and 
Wall index, just to 
mention only a few of them) were constructed. We refer to  \cite{RS93}, 
\cite{CLM94},
\cite{Lon02} and references therein for an exhaustive account on the subject. In 
the present paper, 
we shall use the construction given in \cite{CLM94} (cf. also \cite{RS93}), 
where the authors associate a Maslov 
index to any ordered pair of continuous (and piecewise smooth) paths of 
Lagrangian subspaces, by assigning  an integer which heuristically counts 
algebraically (namely with 
signs and multiplicities) the number 
of non-trivial intersections  between the paths of Lagrangian subspaces. 

The paper is organised as follows.  Section \ref{sec:intro} is devoted to 
describe the problem,  to 
introduce the building blocks of the index theory constructed in the paper as 
well as to state and to 
give an account of the ideas behind the proof of the main results. Section 
\ref{sec:spectral-flow-and-Maslov} 
is dedicated to recall as well as to describe the main properties of the Maslov 
index for pairs as well 
as of the spectral flow for paths of closed selfadjoint Fredholm operators, 
which are behind the notion of the 
geometrical and the spectral index, respectively.  Section 
\ref{sec:spectral-flow-formula}, which represents the 
core of the paper, is devoted to prove the main results stated in Section 
\ref{sec:intro} whose proofs are 
scattered along the whole of Section.

\section{Description of the problem and main results}\label{sec:intro}

The goal of this section is to introduce the dynamical framework in order to 
describe the 
problem, to introduce the main definitions and  ingredients of the index theory 
both in the unbounded and bounded case. 
We conclude this section by stating the main results and discussing 
the principal consequences of the index theory constructed in the paper.

\subsection{Index theory for unbounded orbits}
Given the  $\mathscr C^2$-Hamiltonian function  
$H: \R \times \R^{2n} \to \R$, we 
start to consider the  Hamiltonian system 
\begin{equation}\label{eq:intro-nonlin-hamsys-intro-single}
 w'(t)=J\, \nabla H\big(t,w(t)\big), \qquad   t \in  \R
\end{equation}
where $'$ denotes the derivative with respect to the (time) $t$ variable, $ 
\nabla $  the gradient with 
respect to  the second variable and where $J$ is the standard complex structure
\[
 J= \begin{bmatrix}
     0 & -\Id_n\\
     \Id_n & 0
    \end{bmatrix}.
\]
We assume that $p, q \in \R^{2n}$  are two {\em restpoints\/} (or {\em 
equilibria\/}) for 
the Hamiltonian vectorfield; thus in particular
\[
\nabla H(t, p)= \nabla H(t, q)=0  \qquad \forall \, t \in \R.
\]
\begin{figure}[ht!]
 \centering
\includegraphics[scale=0.20]{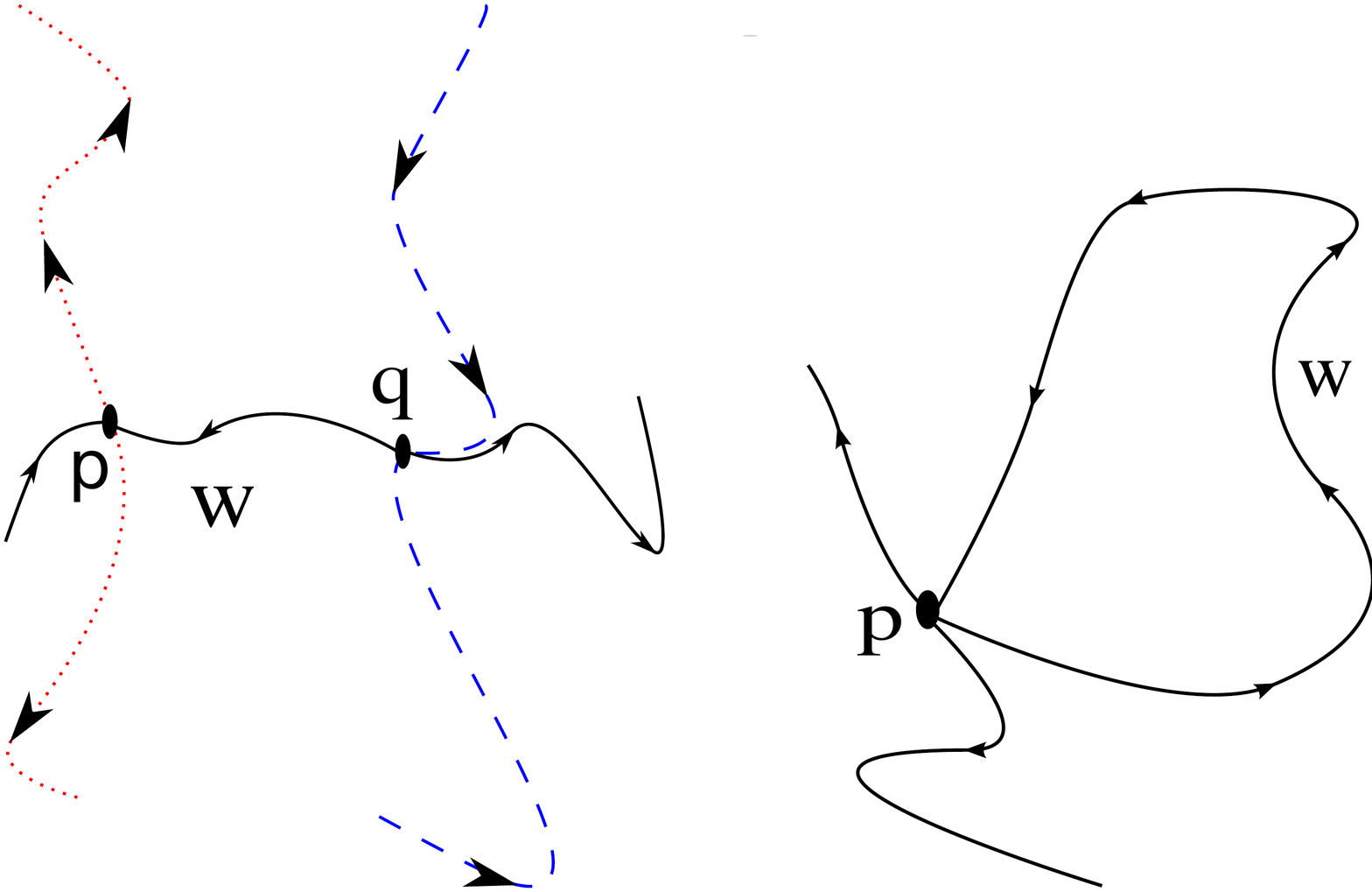}
% % \includegraphics[scale=0.2]{heteroclinic.pdf}
\caption{On the left  is sketched a heteroclinic connection 
$w$ 
between the two saddle points 
$p$ and $q$. The (red) punctured  curve represents the unstable manifold of the 
point $p$ 
and the (blue) dashed curve 
represents the stable manifold of the point $q$. 
On the right is drawn a homoclinic orbit at  $p$.}\label{fig:connections}
\end{figure}
In the rest of the paper, we  
always assume that they are {\em hyperbolic\/}, in the sense that the spectrum 
of the  Hessian 
matrix of $H$ at $p$ and $q$ is off  the imaginary axis, namely
\[
\sigma\big( D^2 H(\cdot,p)\big) \cap i \R = \sigma\big( D^2 H(\cdot,q)\big) \cap 
i \R = \emptyset.
\]
A {\em heteroclinic connection between  $p$ and $q$\/} is 
a solution  of Equation \eqref{eq:intro-nonlin-hamsys-intro-single}  
satisfying the following asymptotic (boundary) conditions
\begin{equation}\label{eq:intro-bc-hetero-intro}
 \lim_{t\to-\infty}w(t)=p \quad \textrm{ and } \quad  \lim_{t \to 
+\infty}w(t)=q. 
\end{equation}
In the particular case where the orbit is asymptotic both in the past and in 
the future  to the same equilibrium point, namely
\begin{equation}\label{eq:intro-bc-homo-single}
\quad \lim_{|t|\to+\infty}w(t)=p\quad \Big(\textrm{or } \quad  \lim_{|t| \to 
+\infty}w(t)=q\Big),
\end{equation}
we shall refer to $w$ as {\em homoclinic solution at $p$ \/} (or $q$), 
respectively.

The other class of unbounded motions that we  introduce,  are termed 
{\em future halfclinic solutions\/} (resp. {\em past halfclinic solutions\/}).  
Let $L$  be a  Lagrangian  subspace and let $H: [0,+\infty) \times \R^{2n} \to 
\R$ be a $\mathscr C^2$-function. 
A  {\em future halfclinic solution at $q$\/} is a  solution of the  asymptotic 
boundary value problem
\begin{equation}\label{eq:intro-bc-half-future-single}
 \begin{cases}
w'(t)=J\, \nabla H\big(t,w(t)\big), \qquad   t \in  (0,+\infty)\\
w(0) \in L ,\ \lim_{t \to +\infty}w(t)=q.
 \end{cases}
 \end{equation}
 \begin{figure}[ht]
\centering
\includegraphics[scale=0.20]{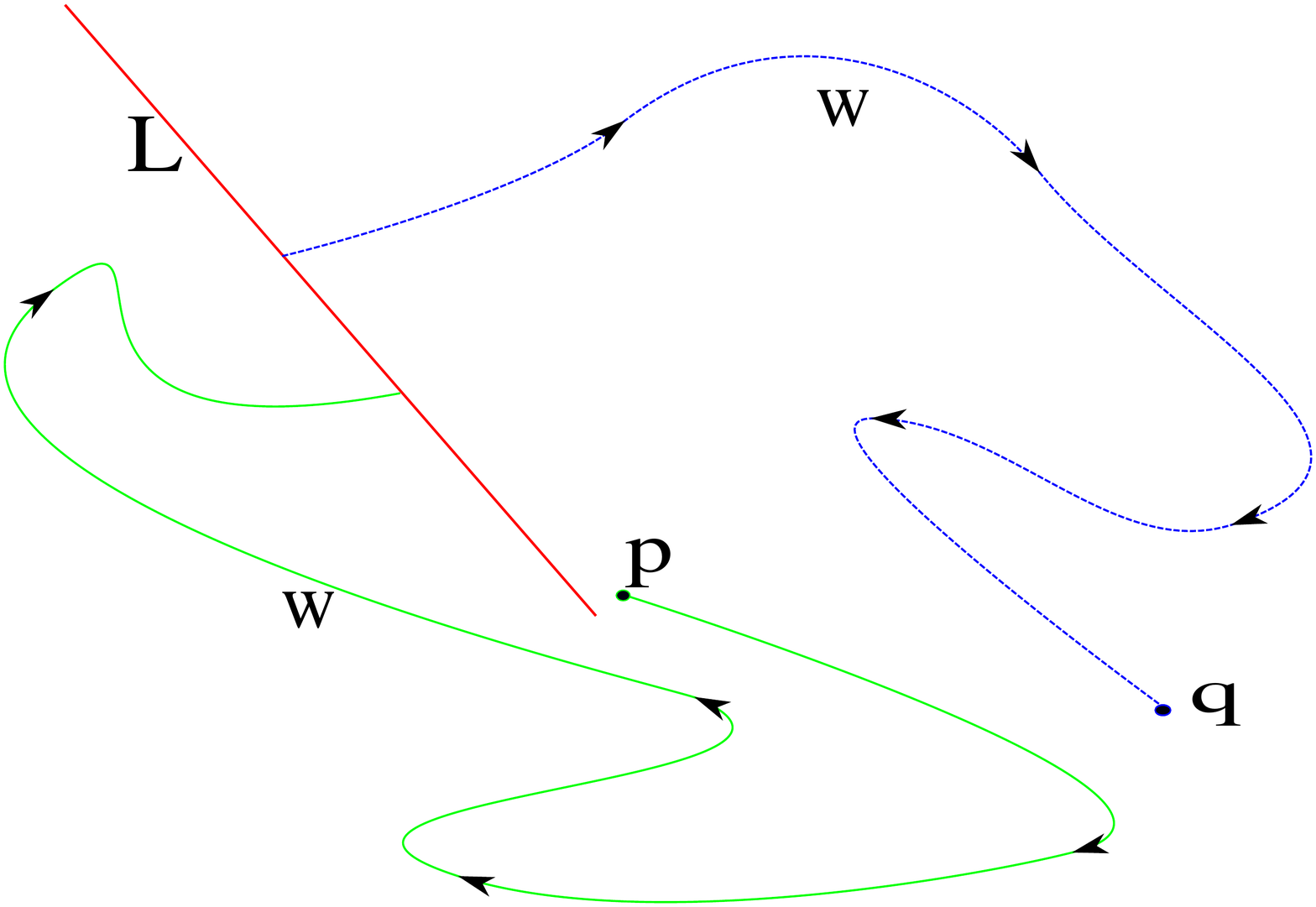}
\caption{A past halfclinic orbit (solid green line) between $p$ and $L$ 
and a future halfclinic orbit (dashed blu line) between $L$ and 
$q$.}\label{fig:halfline}
\end{figure}
Analogously, if  $H: (-\infty,0] \times \R^{2n} \to \R$ is of class $\mathscr 
C^2$, 
we define a  {\em past halfclinic solution at $p$\/} as a  
solution of the following asymptotic boundary value problem
\begin{equation}\label{eq:intro-bc-half-past-single}
  \begin{cases}
w'(t)=J\, \nabla H\big(t,w(t)\big), \qquad   t \in  (-\infty,0)\\
  \lim_{t \to -\infty}w(t)=p,\   w(0) \in L.\  
 \end{cases}
\end{equation}
By linearising Equation \eqref{eq:intro-nonlin-hamsys-intro-single} along  the 
{\em h-clinic 
solution\/}  $w$ (where {\em h-\/} stands for hetero or homo- or half), we 
end-up with the 
following linear Hamiltonian system 
\begin{equation}\label{eq:Ham-sys-bvp-het-hom-intro-single}
\begin{cases}
  z'(t)=J B(t)\, z(t), \qquad t \in\R\\
  \lim_{t \to +\infty} z(t)=0  \quad  \lim_{t \to -\infty} z(t)=0 
 \end{cases}
\end{equation}
in the heteroclinic/ homoclinic case and with
\begin{equation}\label{eq:Ham-sys-bvp-half-intro-single}
\begin{cases}
  z'(t)=J B(t)\, z(t), \qquad t \in [0,+\infty) \\
   z(0) \in L \textrm{ and } \lim_{t \to +\infty} z(t)=0 
 \end{cases}\quad \left(\textrm{resp. }
 \begin{cases}
  z'(t)=J B(t)\, z(t),\qquad t \in (-\infty,0] \\
  \lim_{t \to -\infty} z(t)=0 \textrm{ and }  z(0) \in L
 \end{cases}\right)
\end{equation}
in the future (resp. past) halfclinic case.
Denoting by  $\gamma_{\tau}$ the matrix solution of the 
Hamiltonian initial   value problem given in Equation 
\eqref{eq:intro-nonlin-hamsys-intro-single} 
such that $\gamma_{\tau}(\tau)=\Id$,  
we recall the  {\em stable  \/}  and the {\em unstable  subspaces\/} are 
respectively given by 
\begin{equation}\label{eq:stabili-intro-single}
  E^s(\tau)\=\Set{v \in \R^{2n}|\lim_{t\to +\infty} 
\gamma_{\tau}(t)\,v=0} \textrm{ and }
   E^u(\tau)\=\Set{v \in \R^{2n}|\lim_{t\to -\infty} 
\gamma_{\tau}(t)\,v=0}.
\end{equation}
Let us now consider the constant  solutions of the Hamiltonian system given in 
Equation \eqref{eq:intro-nonlin-hamsys-intro-single} at the restpoints 
$p,q$. Denoting by $B(-\infty)$ and by 
$B(+\infty)$ the 
linearisation of $\nabla H$ at (the constant solutions) $p$ and $q$ 
respectively, we get
\begin{equation}\label{eq:asympt-linear-system-intro-single}
   z'(t)=JB(\pm \infty)\, z(t), \qquad t \in \R\\
\end{equation}
and, by  the hyperbolicity assumption on $p$ and $q$, we get  
$\sigma\big(JB(\pm \infty)\big) \cap i \R= \emptyset$. We  let 
$S(t)\=JB(t)$,  $S(\pm \infty)\=JB(\pm \infty)$, 
and we set
\begin{multline}
 E^s(\pm \infty)\=\Set{v \in \R^{2n}|\lim_{t\to +\infty} \exp\big(t 
S(\pm\infty)\big)v=0} 
 \textrm{ and }\\
   E^u(\pm \infty)\=\Set{v \in \R^{2n}|\lim_{t\to -\infty}\exp\big(t 
S(\pm\infty)\big)v=0}.
\end{multline}\br
%Under the assumption (H0),  
The invariant stable and unstable subspaces 
defined above are Lagrangian subspaces (cf., for instance, \cite{CH07, Wat15} 
and references therein) and  the following convergence result holds:
\begin{equation}\label{eq:convergence-single}
 \lim_{\tau \to +\infty}E^s(\tau)=E^s(+\infty) \quad
\textrm{ and } \quad  
 \lim_{\tau \to -\infty} E^u(\tau)= E^u(-\infty).
 \end{equation}
(Cf. \cite[Proposition 1.2]{AM03}). 
Following authors in \cite{CLM94} to each  ordered pair of Lagrangian paths 
\[
\tau \longmapsto \big(E^s(\tau),E^u(-\tau) \big),  \quad \tau \longmapsto 
\big(E^s(\tau),L\big) 
\quad \textrm{and finally } \tau \longmapsto \Big(L,E^u(-\tau)\Big)
\]
we can assign  an integer known as {\em 
Maslov index of the pair\/} $\iCLM$ which  heuristically counts  the 
nontrivial intersections (with sign) between the paths defining the pair 
when the parameter $\tau$  varies (Cf. Section 
\ref{sec:spectral-flow-and-Maslov}, for the definition). 

\begin{rem}
Several constructions for the Maslov index are available 
in the literature, but in this paper we essentially follows the ones given by 
authors in 
\cite{Arn67,RS93,CLM94, ZL99} and  especially by Cappell, Lee \& Miller in 
\cite{CLM94}. 
\end{rem}
We are now ready to associate to any h-clinic solution, the {\em geometrical 
index\/}. 
\begin{defn}\label{def:geometrical-index}
We define the {\em geometrical index\/} of the 
\begin{itemize}
\item {\em  heteroclinic\/} or {\em homoclinic orbit $w$\/}, as the integer 
given by  
 \[
  \igeo(w)\= -\iCLM\Big(E^s(\tau),E^u(-\tau); \tau \in [0,+\infty)\Big);
 \]
 \item  {\em  future halfclinic orbit $w$\/} between $L$ and $q$, as the integer 
given by 
\[
  \igeo(w)\= -\iCLM\Big(E^s(\tau),L; \tau \in [0,+\infty)\Big);   
\]
\item {\em  past halfclinic orbit $w$\/} between $p$ and $L$, as  the integer 
defined by
\[
  \igeo(w)\= -\iCLM\Big(L,E^u(-\tau); \tau \in [0,+\infty)\Big).
\]
\end{itemize}
\end{defn}
% %  \begin{defn}\label{def:nondeg} 
% %  The heteroclinic orbit $w$ is termed non-degenerate or {\em admissible} if
% %  $E^u(0)\cap E^s(0)=0$. It is termed   {\em boundary non 
% % degenerate\/} (BND, for short)  if  $E^u(-\infty)\cap E^s(+\infty)=0$.
% % Simillary, the
% %  future (resp. past) halfclinic orbit $w$ is termed  {\em boundary non 
% % degenerate\/}   if 
% %  \begin{equation}\label{eq:nondeg-half}
% %   L \cap E^s(+\infty) = 0 \quad \big(\textrm{resp. }  
% %   L_\lambda \cap E^u_\lambda(-\infty) = 0\big).  
% %  \end{equation}
% % It is termed  {\em admissible\/} if 
% %  \begin{equation}\label{eq:admissible-half}
% %   L\cap E^s(0) = 0 \quad (\textrm{resp. }  L 
% % \cap E^u(0) = 0).  
% %  \end{equation}
% % \end{defn}
Let now $H:[0,1] \times \R \times \R^{2n}\to \R$ be a continuous map such that 
$H_\lambda\=H(\lambda, \cdot,\cdot): \R \times \R^{2n} \to \R$ is of class 
$\mathscr C^2$
for all $\lambda \in [0,1]$ 
and its derivatives depend continuously on $\lambda$.  We consider the 
one-parameter family of 
Hamiltonian systems 
\begin{equation}\label{eq:intro-nonlin-hamsys-intro}
 w'(t)=J\, \nabla H_\lambda\big(t,w(t)\big), \qquad   t \in  \R.
\end{equation}
For each $\lambda \in [0,1]$, let $p_\lambda, q_\lambda \in \R^{2n}$ such that
\[
 \nabla H_\lambda \big(t,p_\lambda\big)=\nabla H_\lambda 
\big(t,q_\lambda\big)=0, 
 \qquad \forall\, (\lambda,t) \in [0,1] \times \R
\]
we assume that they are  hyperbolic restpoints and let us 
denote by $w_\lambda$ the heteroclinic connection between $p_\lambda$ and 
$q_\lambda$. Furthermore,
let $\lambda\mapsto 
L_\lambda$  be a continuous path of  Lagrangian 
subspaces and we denote by $w_\lambda$ the past halfclinic connection between 
$p_\lambda$ and $L_\lambda$ and 
by $w_\lambda$ the future halfclinic connection between $L_\lambda$ and 
$q_\lambda$. 
By linearising  Equation \eqref{eq:intro-nonlin-hamsys-intro} along  the {\em 
h-clinic 
solution\/}  $w_\lambda$ we get the 
following linear Hamiltonian system 
\begin{equation}\label{eq:Ham-sys-bvp-het-hom-intro}
\begin{cases}
  z'(t)=J B_\lambda(t)\, z(t) \qquad t \in\R\\
  \lim_{t \to +\infty} z(t)=0  \quad  \lim_{t \to -\infty} z(t)=0 
 \end{cases}
\end{equation}
in the heteroclinic/ homoclinic case and 
\begin{equation}\label{eq:Ham-sys-bvp-half-intro}
\begin{cases}
  z'(t)=J B_\lambda(t)\, z(t) \qquad t \in [0,+\infty) \\
   z(0) \in L_\lambda \textrm{ and } \lim_{t \to +\infty} z(t)=0 
 \end{cases}\quad \left(\textrm{resp. }\begin{cases}
  z'(t)=J B_\lambda(t)\, z(t)\qquad t \in (-\infty,0] \\
  \lim_{t \to -\infty} z(t)=0 \textrm{ and }  z(0) \in L_\lambda 
 \end{cases}\right)
\end{equation}
in the future (resp. past) halfclinic case.
Denoting by  $\gamma_{(\tau,\lambda)}$ the  matrix solution of the 
Hamiltonian system given in Equation \eqref{eq:intro-nonlin-hamsys-intro}, 
such that $\gamma_{(\tau,\lambda)}(\tau)=\Id$,  
we recall the  {\em stable \/}  and the {\em unstable  subspaces\/} are 
respectively given by 
\begin{multline}\label{eq:stabili-intro}
  E^s_{\lambda}(\tau)\=\Set{v \in \R^{2n}|\lim_{t\to +\infty} 
\gamma_{(\tau,\lambda)}(t)\,v=0} \textrm{ and }\\
   E^u_{\lambda}(\tau)\=\Set{v \in \R^{2n}|\lim_{t\to -\infty} 
\gamma_{(\tau,\lambda)}(t)\,v=0}.
\end{multline}
Let us now consider the constant  solutions of the Hamiltonian system given in 
Equation 
\eqref{eq:intro-nonlin-hamsys-intro} at the restpoints 
$p_\lambda,q_\lambda$. By linearising along them and denoting by 
$B_\lambda(-\infty)$ and by 
$B_\lambda(-\infty)$ the 
linearisation of $\nabla H_\lambda$ along the (constant solutions) $p_\lambda$ 
and $q_\lambda$ 
respectively, we get
\begin{equation}\label{eq:asympt-linear-system-intro}
   z'(t)=JB_\lambda(\pm \infty)\, z(t), \qquad t \in \R\\
\end{equation}
and, by the hyperbolicity of $p_\lambda$ and $q_\lambda$, for each $\lambda \in 
[0,1]$, 
$\sigma\big(JB_\lambda(\pm \infty)\big) \cap i \R= \emptyset$. We  let 
$S_\lambda(t)\=JB_\lambda(t)$ and $S_\lambda(\pm \infty)\=JB_\lambda(\pm 
\infty)$.
Let us consider 
the  continuous two-parameters family of Hamiltonian matrices   $S:[0,1] \times 
\R \to \Mat(2n, \R)$ and 
assume that:
\begin{itemize}
\item[{\bf (H1)\/}] There exists two continuous paths that will be denoted by 
$\lambda \mapsto 
S_\lambda(+\infty)$ and 
$\lambda \mapsto S_\lambda(-\infty)$ such that
\[
S_\lambda(+\infty)=\lim_{t\to +\infty}S(\lambda,t) \quad \textrm{ and }  \quad
S_\lambda(-\infty)=\lim_{t\to -\infty} S(\lambda, t), \qquad \lambda\in [0,1]
\]
uniformly with respect to $\lambda$. Moreover  for every $\lambda \in [0,1]$
the matrices $S_\lambda(\pm \infty)$ are hyperbolic. 
\end{itemize}
We set
\begin{multline}
 E_\lambda^s(\pm \infty)\=\Set{v \in \R^{2n}|\lim_{t\to +\infty} \exp\big(t 
S_\lambda(\pm\infty)\big)v=0} 
 \textrm{ and }\\
   E_\lambda^u(\pm \infty)\=\Set{v \in \R^{2n}|\lim_{t\to -\infty}\exp\big(t 
S_\lambda(\pm\infty)\big)v=0}.
\end{multline}\br
Under the assumption (H1),  
the invariant stable and unstable subspaces 
defined above are Lagrangian subspaces (cf. Lemma \ref{thm:stable-lagrangians}) 
and by invoking \cite[Proposition 1.2]{AM03}, the following convergence result holds
\begin{equation}\label{eq:convergence}
 \lim_{\tau \to +\infty}E_\lambda^s(\tau)=E_\lambda^s(+\infty) \quad
\textrm{ and } \quad  
 \lim_{\tau \to -\infty} E_\lambda^u(\tau)= E_\lambda^u(-\infty).
 \end{equation}

To the linearised Hamiltonian system given in Equations 
\eqref{eq:Ham-sys-bvp-het-hom-intro} and \eqref{eq:Ham-sys-bvp-half-intro}, we 
associate respectively the  
closed selfadjoint Fredholm operators  
\begin{multline}\label{eq:cammino-intro}
A_\lambda\=  -J \dfrac{d}{dt}- B_\lambda(t) \subset  \mathscr D(A_\lambda) 
\subset L^2(\R;\R^{2n}) 
\to L^2(\R;\R^{2n})\textrm{ and }\\
A^+_\lambda\=  -J \dfrac{d}{dt}- B_\lambda(t) \subset  \mathscr D(A^+_\lambda) 
\subset L^2([0,+\infty);\R^{2n}) 
\to L^2([0,+\infty);\R^{2n})\\
\left(\textrm{resp. } 
A^-_\lambda\=  -J \dfrac{d}{dt}- B_\lambda(t) \subset  \mathscr D(A^-_\lambda) 
\subset L^2((-\infty,0];\R^{2n}) 
\to L^2((-\infty,0];\R^{2n})\right)
\end{multline}
where $\mathscr D(A_\lambda)=W^{1,2}(\R;\R^{2n})$ and $\mathscr 
D(A_\lambda^\pm)= W_\lambda^\pm$ for 
\begin{multline}\label{eq:w+o-intro}
W_\lambda^+\=\Set{u \in W^{1,2}([0,+\infty); \R^{2n})| u(0) \in L_\lambda}
 \textrm{ and } \\
 W_\lambda^-\=\Set{u \in W^{1,2}((-\infty,0]; \R^{2n})| u(0) \in L_\lambda}.
\end{multline}
Denoting by $\cfsa(V)$ the space of all closed selfadjoint Fredholm operators 
on the real separable 
Hilbert space $V$,  we have three well-defined (gap) continuous paths of closed 
selfadjoint Fredholm  operators
\begin{multline}
A:[0,1] \to \cfsa\big(L^2(\R, \R^{2n})\big) \textrm{ and }\\ 
A^+:[0,1] \to \cfsa\big(L^2([0,+\infty), \R^{2n})\big) \Big(\textrm{resp. } 
 A^-:[0,1] \to \cfsa\big(L^2((-\infty,0], \R^{2n})\big)\Big). 
 \end{multline}
Now, if $p_\lambda,q_\lambda$ are two families of  hyperbolic restpoints and 
$w_\lambda$ 
is a heteroclinic between 
them, we define the {\em spectral index of $w_\lambda$\/} as minus the spectral 
flow of $A$, i.e. 
\[
\ispec(w_\lambda; \lambda \in [0,1])\= -\spfl(A_\lambda; \lambda \in [0,1])
\]
and 
analogously if $L \in \mathscr C^0\big([0,1]; \Lagr(n)\big)$ and $w_\lambda$ is  
either 
a future  halfclinic between $L_\lambda $ and $p$ (resp. a past halfclinic 
between $q_\lambda$ and 
$L_\lambda$), we define the 
spectral index of $w_\lambda$ as the negative spectral flow of the paths $A^+$ 
or 
$A^-$, 
respectively; thus 
\[
\ispec(w_\lambda)\= -\spfl(A_\lambda^+; \lambda \in [0,1]) \quad 
\big(\textrm{resp. } 
\ispec(w_\lambda)\= -\spfl(A_\lambda^-; \lambda \in [0,1])\big) .
\]
(Cf. Sections \ref{sec:spectral-flow-and-Maslov} and Section 
\ref{sec:spectral-flow-formula} for 
further details).   Under this notation, the main result of this paper reads as 
follows. 
\begin{mainthm}{\bf (Index Theorem for families of h-clinic 
orbits)\/}\label{thm:main1-intro}
Let $p_\lambda,q_\lambda$ be two hyperbolic restpoints of the Hamiltonian system 
given in 
Equation \eqref{eq:intro-nonlin-hamsys-intro} 
and let $w_\lambda$ be a heteroclinic connection between them. Under the 
previous notation and if 
(H1) holds, we get 
\begin{itemize}
\item[]{\bf(heteroclinic/homoclinic case )\/}
\[
\begin{split}
 \ispec(w_\lambda;\lambda \in[0,1])=\igeo(w_1) - \igeo(w_0) + 
 \iCLM\big(E^s_\lambda(+\infty),E^u_\lambda(-\infty); \lambda \in [0,1]\big).
\end{split}
\]
\end{itemize}
Let $L \in \mathscr C^0\big([0,1]; \Lagr(n)\big)$ and let $w_\lambda$ either a 
future 
halfclinic solution between $L_\lambda$ and $q_\lambda$ or  
a past halfclinic solution between $p_\lambda$ and $L_\lambda$. Under the 
previous notation and if 
(H1) is fulfilled, then we have 
\begin{itemize}
\item[]  {\bf(future halfclinic case)\/}
\[
\begin{split}
\ispec(w_\lambda;\lambda \in[0,1])=\igeo(w_1) - \igeo(w_0)  + 
 \iCLM\big(E^s_\lambda(+\infty),L_\lambda; [0,1]\big)
\end{split}
\]
\item[]  {\bf(past halfclinic case)\/}
\[
\begin{split}
\ispec(w_\lambda;\lambda \in [0,1])=\igeo(w_1) - \igeo(w_0)  + 
 \iCLM\big(L_\lambda,E^u_\lambda(-\infty); [0,1]\big).
\end{split}
\]
\end{itemize}
\begin{figure}[ht]
  \centering
\includegraphics[scale=0.30]{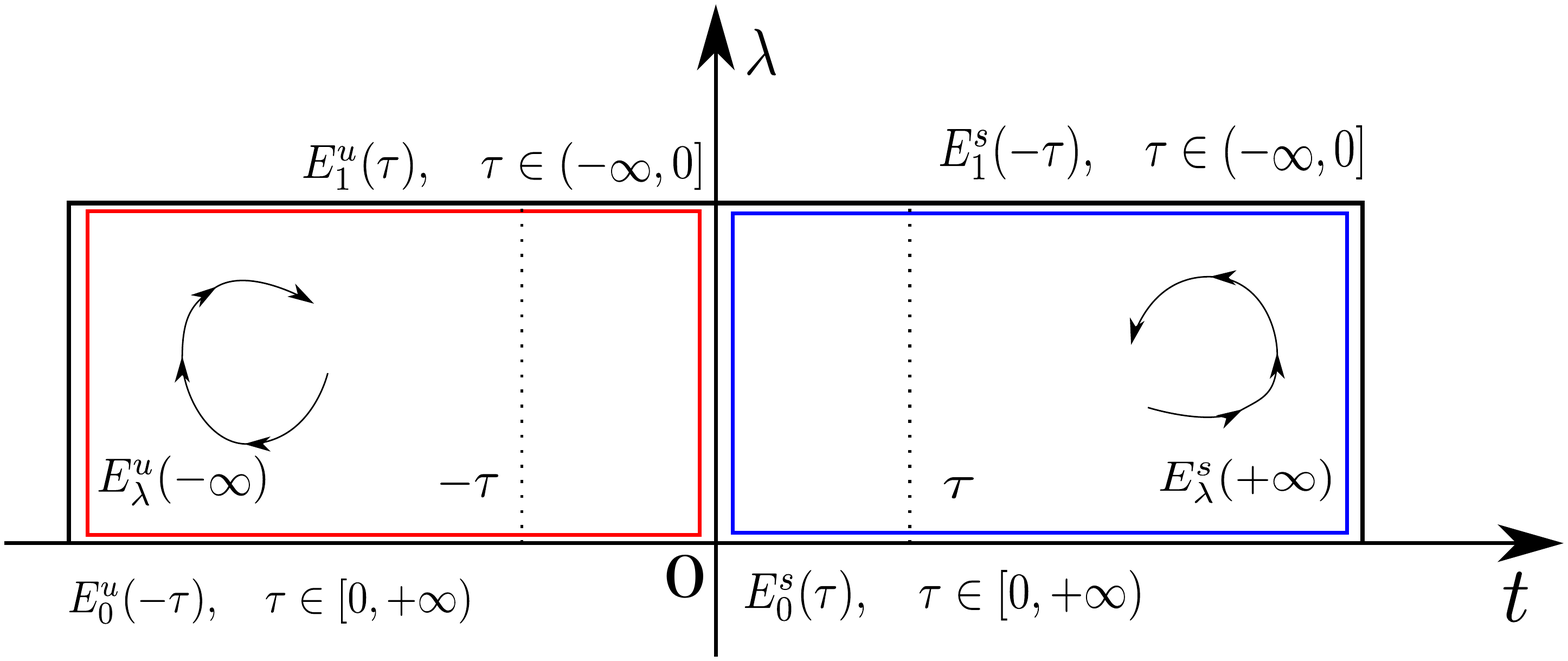}
  \caption{A sketch  of the stable (blue rectangle on the right) and 
unstable subspace 
 (red rectangle on the left) parameter spaces.}\label{fig:stable}
\end{figure} 
\end{mainthm}
\begin{rem}
A few comments on  the proof of Theorem \ref{thm:main1-intro} are in order.  
The idea behind  
the spectral flow formulas in the heteroclinic/homoclinic and halfclinic cases 
relies
on the fact that either the path  $A$ or the paths $A^\pm$ (defined 
in Equation \eqref{eq:cammino-intro}) are stratum-homotopic to the real first 
order  selfadjoint elliptic operator 
\begin{equation}\label{eq:floer-intro}
D(E^u_\lambda(0), E^s_\lambda(0))\phi\=- J \dfrac{d\phi}{dt}
\end{equation}
on the domain $ W_\lambda\=\Set{u \in W^{1,2}([a,b];\R^{2n})| 
u(a) \in E^u_\lambda(0), \ u(b) \in E^s_\lambda(0)}$ and one of the 
main ingredient  of the proof (scattered all along the whole of 
Section \ref{sec:spectral-flow-formula}) is based on the $\varepsilon$-spectral 
flow as well as on \cite[Theorem 0.4]{CLM94} which will be crucial in our 
arguments.  
It is worth noticing that the crucial role of the operator given in Equation 
\eqref{eq:floer-intro} 
could be traced back to the works by  Floer \cite{Flo88a, Flo88b, Flo88c} and 
few years later to 
Yoshida \cite{Y0s91}.
\end{rem}

\begin{rem}
We observe that in the case of future and past  halfclinic orbits, 
the pairs of paths of 
Lagrangian subspaces to be considered are $E=\big(E_\lambda^s(0), 
L_\lambda\big)$ and    
$E=\big(L_\lambda, E_\lambda^u(0)\big)$. 
\end{rem}
As direct application of Theorem \ref{thm:main1-intro}, in the special case of a 
(single) Hamiltonian system, 
we get a spectral flow formula for a homoclinic  and future/past halfclinic  
orbits in terms of the {\em relative Morse index\/}.
(Cf. We refer the interested reader 
to Section \ref{sec:spectral-flow-formula} for the proof). 
\begin{maincor}{\bf (Index Theorem for homoclinic orbits of Chen and 
Hu)\/}\label{thm:main2-intro}
Let $H \in \mathscr C^2\big(\R \times \R^{2n}, \R\big)$ be a Hamiltonian 
function,  $p$ be a hyperbolic restpoint,  $w$ be a (homoclinic)  
solution of the Hamiltonian system
\begin{equation}\label{eq:homo-cor-intro}
\begin{cases}
 w' = J \nabla H\big(t, w(t)\big), \qquad t \in \R\\
 \lim_{|t|\to+\infty} w(t)=p.
\end{cases}
\end{equation}
We set $ B(t)\= D^2 H\big(t,  z(t)\big)$  and $B_*$ 
respectively the linearisation of $\nabla H$ along the homoclinic  orbit $z$ and 
along the constant 
solution $p$ and we assume that 
\[
\lim_{z\to p} D^2 H\big(t, z\big)=B_*
\]
uniformly with respect to $t \in \R$. Then, we have
\begin{equation}\label{eq:ll}
 \irel\left(- J \dfrac{d}{dt}- B_*, -J \dfrac{d}{dt}-  B(t) \right)
 = \igeo(w)
\end{equation}
where  we denoted by $\irel$  the {\em relative Morse index\/}(cf. Definition 
\ref{def:rel-morse-index}).
\end{maincor}
An analogous result holds in the halfclinic case. 
\begin{maincor}{\bf (Index Theorem for halfclinic  
orbits)\/}\label{thm:main3-final-intro}
Let $H \in \mathscr C^2\big([0,+\infty) \times \R^{2n}, \R\big)$ (resp.   
$H \in \mathscr C^2\big((-\infty,0] \times \R^{2n}, \R\big)$ be a Hamiltonian 
function,   $q$ be a hyperbolic restpoint,  $L \in \Lagr(n)$ and  $w$ be a 
future (resp. past)  
halfclinic solution of the Hamiltonian system
\begin{equation}\label{eq:harfe-cor-intro}
\begin{cases}
 w' = J \nabla H\big(t, w(t)\big), \qquad t \in [0,+\infty)\\
w(0) \in L \ \textrm{ and }\  \lim_{t\to+\infty} w(t)=q
\end{cases}
\left(\textrm{resp. }
\begin{cases}
 w' = J \nabla H\big(t, w(t)\big), \qquad t \in (-\infty,0]\\
\lim_{t\to-\infty} w(t)=p\ \textrm{resp.}\ w(0)
\end{cases}\right).
\end{equation}
Under the assumptions of Corollary \ref{thm:main2-intro},   we have  in both 
cases
\begin{equation}
 \irel\left(- J \dfrac{d}{dt}- B_*, -J \dfrac{d}{dt}-  B(t) \right)=\igeo(w).%\\
% \irel\left(- J \dfrac{d}{dt}- B_*, -J \dfrac{d}{dt}-  B(t) \right)=- \dim 
%\big(L\cap E_*^u\big)  +\dim \big(L \cap E^u(0)\big)+\igeo(w).
\end{equation}
\end{maincor}

\subsection{A new index theory for bounded motions}

This subsection is devoted to provide a new spectral flow  formula for 
Hamiltonian 
systems defined on  bounded intervals.
Let $L, M \in \Lagr(n)$, $a,b \in \R$ and $H:[a,b]\times \R^{2n} \to \R$ be a 
$\mathscr C^2$-function 
and let $w$ be a solution of the Hamiltonian system 
\begin{equation}\label{eq:Hamiltonian-bounded-intro}
\begin{cases}
 w'(t)=J\nabla H\big(t, w(t)\big), \qquad t \in (a,b)\\
 w(a) \in L \ \textrm{ and } w(b) \in M.
\end{cases}
 \end{equation}
We denote by  $\gamma_{a}$ the (fundamental) matrix-valued solution 
$\gamma_{a}: [a,b] \to \Sp(2n,\R)$  of the Hamiltonian system given in 
Equation \eqref{eq:Hamiltonian-bounded-intro}
\begin{equation}\label{eq:cauchy-bounded-intro}
 \begin{cases}
   \gamma'_{a}(t)= S(t)\,\gamma_{a}(t), \qquad t 
\in [a,b]\\
  \gamma_{a}(a)=\Id.
 \end{cases}
\end{equation}
For any $ c \in (a,b)$ we consider the ordered pair
of Lagrangian paths pointwise defined by 
\[
\big( W(c),V(c)\big)
\]
where 
\begin{equation}\label{eq:WeV}
W(c)\=\gamma^{-1}_a(b-c)M
\textrm{ and } V(c)\=\gamma_{a}(c)L.
\end{equation}
Let $\alpha:[a,b] \to [c,b]$ be a positive affine reparametrisation of 
the interval $[c,b]$ (with the same  orientation) whilst 
the function $\beta:[a,b] \to [a,c]$ is  a negative affine reparametrisation of 
the interval  $[a,c]$ (with the opposite orientation). 
\begin{defn}\label{def:geometrical-index-bounded}
Under the previous notation, we define the {\em geometrical index\/} of the  
{\em bounded solution $w$\/} of the Hamiltonian system given in Equation 
\eqref{eq:Hamiltonian-bounded-intro} 
as follows
\[
  \igeob(w)\=\iCLM\Big(W\big(\beta(\tau)\big),V\big(\alpha(\tau)\big); \tau \in 
[a,b]\Big).
\]
\end{defn}
\begin{rem}
It is worth noticing that Definition \ref{def:geometrical-index-bounded} is new 
in the case of bounded 
orbits. In fact, it is standard (cfr. \cite{Arn86,RS95}, for instance) in this 
case to define the 
geometrical index of the solution $w$ of the Hamiltonian system 
given in Equation \eqref{eq:Hamiltonian-bounded-intro} as follows
\[
 \iCLM\big(M, \gamma_a(t)L; \tau \in [a,b]\big).
\]
However, we will prove in Lemma \ref{thm:classical-maslov-comparison} that these 
two integers coincide. In this 
way we are able to recover all the existing results. 
\end{rem}
Let now $H:[0,1] \times [a,b] \times \R^{2n}\to \R$ be a continuous map such 
that 
$H_\lambda\=H(\lambda, \cdot,\cdot): \R \times \R^{2n} \to \R$ is of class 
$\mathscr C^2$ for all $\lambda \in [0,1]$ 
and its derivatives depend continuously on $\lambda$. As before, for each 
$\lambda\in[0,1]$, let $\lambda \mapsto L_\lambda$  and $\lambda\mapsto 
M_\lambda$ be two paths of 
Lagrangian subspaces. We consider the one-parameter family of Hamiltonian 
systems 
\begin{equation}\label{eq:intro-nonlin-hamsys-intro-bd}
\begin{cases}
w'(t)=J\, \nabla H_\lambda\big(t,w(t)\big), \qquad   t \in  (a,b)\\
 w(a)\in L_\lambda \ \textrm{ and } \ w(b) \in M_\lambda
\end{cases}
 \end{equation}
and for each $\lambda \in [0,1]$, we denote by $w_\lambda$ the bounded solution 
of the Hamiltonian system given in Equation 
\eqref{eq:intro-nonlin-hamsys-intro-bd}
between $L_\lambda$ and $M_\lambda$. 
By linearising  Equation \eqref{eq:intro-nonlin-hamsys-intro-bd} along 
$w_\lambda$, we get the 
following linear Hamiltonian system 
\begin{equation}\label{eq:Ham-sys-bvp-het-hom-intro-bdd}
\begin{cases}
  z'(t)=J B_\lambda(t)\, z(t), \qquad t \in(a,b)\\
  z(a) \in L_\lambda \ \textrm{ and }  \ z(b) \in M_\lambda.
 \end{cases}
\end{equation}
We denote by $ W^{1,2}([a,b];L_\lambda, M_\lambda)$ the Sobolev space defined 
by 
\[
 W^{1,2}([a,b],L_\lambda, M_\lambda)\=
 \Set{z \in W^{1,2}([a,b];\R^{2n})| z(a)\in L_\lambda \textrm{ 
and } z(b) \in  M_\lambda}
\]
and for each $\lambda \in [0,1]$, we define  the operators
\begin{equation}\label{eq:operatori-bounded-intro}
 T_\lambda\= - J \frac{d}{dt} - B_\lambda(t):  W^{1,2}([a,b];L, M)
\subset L^2([a,b], \R^{2n})
 \longrightarrow L^2([a,b], \R^{2n}).
\end{equation}
It is well-known that for each $\lambda \in [0,1]$, 
$ T_\lambda$  is a closed unbounded 
selfadjoint Fredholm operator in $L^2([a,b], \R^{2n})$ having domain 
$  W^{1,2}([a,b];L_\lambda, M_\lambda)$  (cf. \cite{GGK90}, for instance) 
and hence it 
remains well-defined (gap continuous) path of closed selfadjoint Fredholm 
operators
\begin{equation}\label{eq:path-fredholm-2-intro}
 T:[0,1]\ni \lambda \longmapsto T_\lambda \in  \cfsa(L^2([a,b],\R^{2n})).
\end{equation}
We define $\ispec(T)\=-\spfl(T_\lambda;\lambda\in[0,1])$.
\begin{mainthm}\label{thm:main-2-intro}{\bf (Spectral flow formula for bounded 
orbits) \/}
Under the previous notation, we have
\begin{equation}\label{eq:forse-spicciamu}
 \ispec(T)=\igeob(w_1)-\igeob(w_0) + 
 \iCLM\big(M,L; [0,1]\big).
\end{equation}
\end{mainthm}
In the special case in which the Lagrangian boundary conditions are independent 
on $\lambda$ as direct consequence 
of the nullity property of the $\iCLM$ index (cf. Section 
\ref{sec:spectral-flow-and-Maslov}) the third term in the 
(RHS) of Equation \eqref{eq:forse-spicciamu}, vanishes identically.
\begin{cor}\label{thm:mainbounded-intro}
 In the assumption of Theorem \ref{thm:main-2-intro}, if $(L_\lambda, 
M_\lambda) 
\equiv (L,M) \in \Lagr(n)\times \Lagr(n)$, 
 then we have
\begin{equation}
 \ispec(T)=\igeob(w_1)-\igeob(w_0).
\end{equation}
\end{cor}
\begin{rem}
We conclude the overview of the main results of this paper by stressing the 
fact that, with this approach, we are able to get  {\em at once\/} the spectral 
flow formulas   for solutions parametrised by a half-line, on the whole real 
line or  
by a  bounded interval, by establishing in the last case a  
precise and net way to relate this new definition to the classical one. 
It is worth mentioning that the definition of the Maslov index given in this 
paper, is very flexible and it agrees also in the unbounded case, with the 
definition given by authors in 
\cite[Definition 1, pag. 592]{CH07} and in \cite{HO16, BHPT17}.
A big effort for the 
readability of the paper was made  by trying on the one-hand to shorten as much 
as possible 
all   proofs by avoiding standard and clumsy details; on the other hand 
we attempt  to  be clear and to add precise  references to the existing 
literature as
much as possible. The paper is structured as 
follows. In Section 
\ref{sec:spectral-flow-and-Maslov} we quickly recap the basic definitions and 
properties of the spectral 
flow and of the Maslov index. In Section \ref{sec:spectral-flow-formula} which 
is the core of the paper, we prove a general spectral flow formula for a 
one-parameter family of Hamiltonian 
systems. As direct application we prove a new 
index theorems for homoclinic and half-clinic motions and we finally 
recover all the existing results in the case of orbits parametrised on bounded 
intervals.
\end{rem}

\section{Maslov index and Spectral Flow}\label{sec:spectral-flow-and-Maslov}

This Section is devoted to recall the basic definitions and properties of the 
{\em Maslov index for pairs of Lagrangian subspaces\/} and of the {\em spectral 
flow for paths of closed 
selfadjoint Fredholm operators.\/}. Our basic references for the material 
contained in this section are 
\cite{CLM94, ZL99} and \cite{BLP05} and references therein.

\subsection{The Maslov index for pairs of Lagrangian paths}
Let  $(\R^{2n},\omega)$ be the standard symplectic space where $\omega$ is the 
(standard) symplectic form 
given by 
\begin{equation}
 \omega(u,v) \= \langle J u, v\rangle \quad \textrm{ for }\quad  J\= 
\begin{pmatrix} 
                                                            0&-I_n\\
                                                            I_n &0
                                                           \end{pmatrix}
\end{equation}
where $I_n$ denotes the identity matrix. We denote by 
$\Lagr(n)$  the set of all Lagrangian subspaces of $(\R^{2n},\omega)$
and we  refer to as the {\em Lagrangian Grassmannian of 
$(\R^{2n},\omega)$.\/}  It is 
well-known that the Lagrangian Grassmannian  is a real compact and connected 
analytic  $\frac12 n(n+1)$-dimensional 
submanifold of the Grassmannian manifold of all $n$-dimensional subspaces in 
$\R^{2n}$.
For $a, b \in \R$ with $a<b$, we denote by $\P([a,b];\R^{2n})$  the 
space of all ordered pairs of  continuous  maps of Lagrangian  subspaces
\[
 L: [a,b] \ni t \mapsto L(t)\= \big(L_1(t), L_2(t)\big) \in  \Lagr(n) \times 
\Lagr(n)
\]
equipped with the compact-open topology. Following authors in \cite{CLM94} we 
recall the definition of the  
{\em  Maslov index   for pairs of  Lagrangian subspaces,\/} that will be denoted 
by 
$\iCLM$. (In the notation, CLM stands for Cappell, Lee and Miller).
Loosely  speaking, given the pair $L=(L_1, L_2) \in \P([a,b];\R^{2n})$, this 
index 
counts with signs and multiplicities the number of instants $t \in [a,b]$ that 
$L_1(t) \cap L_2(t) \neq \{0\}$.  
\begin{defn}\label{def:clm-index}
The {\em CLM-index\/}  is the unique integer valued 
function 
\[
 \iCLM:\P([a,b];\R^{2n})\ni L \longmapsto \iCLM(L;[a,b]) \in \Z
\]
satisfying the Properties I-VI given in \cite[Section 1]{CLM94}. 
\end{defn}
\begin{rem}
Authors in \cite{CLM94} defined their index
in any (finite-dimensional) real or complex symplectic vector space. A different 
approach can 
be 
conceived by using the charts of the differential atlas of $\Lagr(n)$ and the 
fundamental groupoid along the  lines given by authors in \cite{GPP04}.

It is worth mentioning that there is a very efficient way to compute $\iCLM$ 
through 
the so-called {\em crossing forms\/} as shown by authors in \cite{RS93, ZL99}. 
We also observe that 
in the special case in which one Lagrangian path of the pair is constant, the 
$\iCLM$-index is closely 
related to the {\em (relative) Maslov index\/} $\iRS$, defined by authors in 
the 
celebrated paper \cite{RS93}. As proved in \cite[Theorem 3.1]{ZL99}, if 
$L=(L_1,L_2) \in \P([a,b];\R^{2n})$, 
then the following relation holds
\[
\iCLM(L_1, L_2;[a,b])=\iRS(L_2, L_1;[a,b]) - \dfrac12 
\big(h_{12}(b)-h_{12}(a)\big)
\]
where $h_{12}(t)\= \dim L_1(t)\cap L_2(t)$. 
(Cf. \cite[Theorem 3.1]{ZL99} for further details).
\end{rem}

\paragraph{Properties of the CLM-index.}
For the sake of the reader we list some  properties of the $\iCLM$ index that 
we shall frequently use along the 
paper.
\begin{itemize}
\item{\bf (Stratum homotopy relative to the ends)\/} Given a continuous map 
\[
L:[0,1]\ni s \mapsto L(s) \in  \P([a,b];\R^{2n}) \textrm{ where } 
L(s)(t)\=\big(L_1(s,t), L_2(s,t)\big)
\]
such that $\dim\big(L_1(s,a) \cap L_2(s,a)\big)$ and $\dim\big(L_1(s,b) \cap 
L_2(s,b)\big)$ are 
both constant, then 
\[
\iCLM\big(L(0);[a,b]\big)= \iCLM\big(L(1);[a,b]\big).
\]
\item{\bf (Path additivity)\/} Let $a,b,c \in R$ with $a<b<c$. If $L\=(L_1, 
L_2) \in \P([a,c];\R^{2n})$, then 
\[
\iCLM(L;[a,c])= \iCLM(L;[a,b])+\iCLM(L;[b,c]).
\]
\item{\bf (Nullity)\/} Given $L=(L_1,L_2) \in \P([a,b];\R^{2n})$ such that 
$\dim\big(L_1(t\cap L_2(t)\big)$ is 
independent on $t$ and $L_1(t) \cap L_2(t)$ varying continuously, then
\[
\iCLM(L;[a,b])=0.
\]
\item{\bf (Reversal)\/} Let  $L=(L_1,L_2) \in \P([a,b];\R^{2n})$. Denote the 
same path travelled in the reverse 
direction in  $\P([-b,-a];\R^{2n})$ by  $\hat L(s)=\big(L_1(-s), L_2(-s)\big)$. 
Then 
\[
\iCLM(\hat L;[-b, -a])= -\iCLM(L;[a,b]).
\]
\end{itemize}

\subsection{The spectral flow for paths of closed selfadjoint Fredholm 
operators}

Given the separable real Hilbert space 
$V$, we denote by  $\mathcal C^{sa}(V)$ 
the set of all (closed) densely  defined and selfadjoint operators $T : 
\mathcal D(T) \subset V \to 
V$ and by $ \cfsa(V)$ the space of all closed selfadjoint and Fredholm 
operators equipped 
with the {\em graph distance topology\/} (or {\em gap topology\/}), namely the 
topology induced by the {\em gap 
metric\/} $
 d_G(T_1, T_2)\=\norm{P_1-P_2}_{\Lin(V)}$
where $P_i$ is the projection onto the graph of $T_i$ in the product space $V 
\times V$ and $\mathscr L(V)$ denotes the Banach space of all 
bounded and linear operators. 
Let $T \in\cfsa(V)$ and for $a,b \notin 
\sigma(T)$, we set
\[
P_{[a,b]}(T)\=\Re\left(\dfrac{1}{2\pi\, i}\int_\gamma 
(\lambda-T^\C)^{-1} d\, \lambda\right)
\]
where $\gamma$ is the circle of radius $\frac{b-a}{2}$ around the point 
$\frac{a+b}{2}$. We recall that if 
$[a,b]\subset \sigma(T)$ consists of  isolated eigenvalues of finite type then 
\begin{equation}\label{eq:range-proiettor}
 \im  P_{[a,b]}(T)= E_{[a,b]}(T)\= \bigoplus_{\lambda \in [a,b]}\ker 
(\lambda -T);
\end{equation}
(cf. \cite[Section XV.2]{GGK90}, for instance) and $0$ either belongs to the 
resolvent set of $T$ or it is an isolated eigenvalue of finite multiplicity.  
The 
next result allow us to  define the spectral flow for continuous paths in 
$\cfsa(V)$.
\begin{prop}\label{thm:cor2.3}
 Let $T_0 \in \cfsa(V)$ be fixed.
 There exists a positive real number $a \notin \sigma(T_0)$ and an 
open neighborhood $\mathscr N \subset  \cfsa(V)$ of $T_0$ in the gap topology 
such that $\pm a \notin 
\sigma(T)$ for all $T \in  \mathscr N$ and the map 
 \[
  \mathscr N \ni T \longmapsto P_{[-a,a]}(T) \in \Lin(V)
 \]
is continuous and the projection $ P_{[-a,a]}(T)$ has constant finite 
rank for all $t \in \mathscr N$.
\end{prop}
\begin{proof}
For the proof of this result we refer the interested reader to 
\cite[Proposition 2.10]{BLP05}.
\end{proof}
Let now $\mathcal A:[a,b] \to \cfsa(V)$ be a  continuous path.  As direct 
consequence 
of Proposition \ref{thm:cor2.3}, for every $t \in [a,b]$ there exists $a>0$ and 
an open 
connected neighbourhood $\mathscr N_{t,a} \subset \cfsa(V)$ of $\mathcal 
A(t)$ 
such that $\pm a \notin \sigma(T)$ for all $T \in \mathscr N_{t,a}$, the map 
$
 \mathscr N_{t,a} \in T \longmapsto P_{[-a,a]}(T) \in \Lin(V)$
is continuous and hence the rank of $\mathcal P_{[-a,a]}(T)$ does not 
depends on $T \in \mathscr N_{t,a}$.
Now let us consider the open covering of the interval $I$ given by the 
pre-images of the neighbourhoods $\mathcal 
N_{t,a}$ through $\mathcal A$ and by choosing a sufficiently fine partition of 
the interval $[a,b]$ having diameter less than the Lebesgue number 
of the covering, we can find  $a=:t_0 < t_1 < \dots < t_n:=b$, 
operators $T_i \in \cfsa(V)$ and  
positive real numbers $a_i $, $i=1, \dots , n$ in such a way the restriction of 
the path $\mathcal A$ on the 
interval $[t_{i-1}, t_i]$ lies in the neighborhood $\mathscr N_{t_i, a_i}$ and 
hence the 
$\dim E_{[-a_i, a_i](\mathcal A_t)}$ is constant for $t \in [t_{i-1},t_i]$, 
$i=1, \dots,n$.
\begin{defn}\label{def:spectral-flow-unb}
The \emph{spectral flow of $\mathcal A$} on the interval $[a,b]$ is defined by
\[
 \spfl(\mathcal A; [a,b])\=\sum_{i=1}^N \dim\,E_{[0,a_i]}(\mathcal A_{t_i})-
 \dim\,E_{[0,a_i]}(\mathcal A_{t_{i-1}}) \in \Z.
\]
\end{defn}
\begin{rem}
The spectral flow as given in Definition \ref{def:spectral-flow-unb} is 
well-defined 
(in the sense that it is independent either on the partition or on the $a_i$) 
and only depends on 
the continuous path $\mathcal A$. (Cf. \cite[Proposition 2.13]{BLP05} and 
references therein). (We also refer  the interested reader to \cite{Wat15} for 
the analogous properties in the case of closed selfadjoint Fredholm operators 
on 
fixed domains). 
\end{rem}
\paragraph{Properties of the Spectral Flow.}
For the sake of the reader we list some  properties of the spectral flow  that 
we shall frequently use in the paper.
\begin{itemize}
 \item {\bf (Stratum homotopy relative to the ends)\/} Given a continuous map 
 \[
  \overline A: [0,1]\to \mathscr C^0\big([a,b];\cfsa(V)\big)\textrm{ where } 
\overline A(s)(t)\=\overline A^s(t)
\]  
such that $\dim \ker \overline A^s(a)$ and  $\dim \ker \overline A^s(b)$ are 
both independent on $s$, then 
   \[
    \spfl(\overline A^0_t; t \in [a,b])=\spfl(\overline A^1_t; t \in [a,b]).
   \]
    \item  {\bf (Path additivity)\/} If $A^1,
A^2\in \mathscr C^0\big([a,b];\cfsa(V)\big) $ are  such that 
$ A^1(b)= A^2(a)$, then
 \[
  \spfl( A^1_t * A^2_t; t \in [a,b]) = \spfl( A^1_t; t \in [a,b])+\spfl( A^2_t; 
t \in [a,b]).
 \]
   \item  {\bf (Nullity)\/} If $ A\in \mathscr C^0\big([a,b];\GL(V)\big)$, 
   then $\spfl( A_t; t \in [a,b])=0$;
   \item{\bf (Reversal)\/}  Denote the 
same path travelled in the reverse 
direction in  $\cfsa(V)$ by  $\widehat A(t)=A(-t)$. 
Then 
\[
\spfl( A_t; t \in [a,b])=-\spfl(\widehat A _t; t \in [-b,-a]).
\]
\end{itemize}
There is an efficient way to compute the spectral flow, through what are called 
{\em crossing 
forms\/}. Let us consider a $\mathscr C^1$-path, which always exists in the 
homotopy class 
(relative to the ends) and let $P_t$ be the orthogonal projector from $V$ to 
$E_0\big(\mathcal A_t\big)$, the kernel of $\mathcal A_t$. When 
$E_0\big(\mathcal A_{t_0}\big)\neq \{0\}$ we term 
the instant $t_0$ a crossing instant. In this case we defined the {\em crossing 
operator\/} 
$\Gamma(\mathcal A,t_0)$ as 
\[
\Gamma(\mathcal A,t_0)\= P_{t_0} \dfrac{\partial}{\partial t}P_{t_0}: 
E_0\big(\mathcal A_{t_0}\big)\to   
E_0\big(\mathcal A_{t_0}\big).
\]
We term the crossing instant $t_0$ {\em regular\/} if the crossing operator 
$\Gamma(\mathcal A,t_0)$ is 
non-degenerate. In this case we define the {\em signature\/} simply as 
\[
 \sgn\big(\Gamma(\mathcal A,t_0)\big)\= \dim E_+\big(\Gamma(\mathcal 
A,t_0)\big)- 
 \dim E_-\big(\Gamma(\mathcal A,t_0)\big),
\]
where $ E_+\big(\Gamma(\mathcal A,t_0)\big)$ (resp. $ E_-\big(\Gamma(\mathcal 
A,t_0)\big)$) denote the positive 
(resp. negative) spectral space of the operator $\Gamma(\mathcal A,t_0)$. We 
assume that all the 
crossings are regular. Then the crossing instants are isolated (and hence on a 
compact interval  are 
in a finite number) and the spectral flow is given by the following formula 
\begin{equation}\label{eq:formula-sf-crossings}
 \spfl(\mathcal A_t; t \in [a,b]) = \sum_{t_0 \in \mathcal S_*} 
\sgn\big(\Gamma(\mathcal A, t_0)\big)- 
 \dim E_-\big(\Gamma(\mathcal A,a)\big)
+ \dim E_+\big(\Gamma(\mathcal A,b)\big)
\end{equation}
where $\mathcal S_*\= \mathcal S\cap (a,b)$ and $\mathcal S$ denotes the set of 
all crossings. 
\begin{rem}
 One can prove that there exists $\varepsilon_0>0$ sufficiently small such that 
 $\spfl(\mathcal A_t; t \in [a,b])= \spfl(\mathcal A_t+ \varepsilon \Id; t \in 
[a,b])$ (where $\Id$ denotes the 
 identity operator on $V$) for 
 $\varepsilon \in [0, \varepsilon_0]$; furthermore for almost every such 
$\varepsilon$ the 
 path $t \mapsto \mathcal A_t+ \varepsilon \Id$ has regular crossings. We refer 
the interested reader to 
 \cite{CH07, HS09, Wat15} and references therein. 
\end{rem}
Given the  continuous path $ A: [a,b]\to \cfsa(V)$,
we denote by $ A^\varepsilon$  the path $ A^\varepsilon: [a,b]\to \cfsa(V)$ 
pointwise 
defined by 
\[
 A^\varepsilon(t)\=  A_t -\varepsilon \Id, \qquad t \in [a,b].
\]
\begin{lem}\label{thm:differnza-spfl}
There exists  $\varepsilon>0$ sufficiently small, such that 
\[
\spfl( A_t; t\in [a,b])= \spfl( A^{ \varepsilon}_t;t \in [a,b])-  
\dim \ker A_a +\dim \ker A_b.
\]
% % \begin{figure}[ht]
% %  \centering
% %  \includegraphics[scale=0.20]{Figure/A-epsilon.pdf}
% % \caption{In this figure is sketched the rectangle $\mathcal R:= 
% % [a,b]\times[0,1]$. By the contraibility of 
% % the rectangle the spectral flow of the 
% % restriction of the family $\widetilde A$ along  $\partial \mathcal R$ 
vanishes 
% % identically.}
% %  \end{figure}
\end{lem}
\begin{proof}
We start by observing that, since $ A_a,  A_b\in \cfsa(V)$, then $0$ belongs 
to 
the resolvent set of $ A_a$ and $  A_b$ or it is an isolated eigenvalue of 
finite multiplicity. 
Thus we can choose $\varepsilon >0$ smaller than  every non-zero eigenvalue of 
$ 
A_a$ and $ A_b$ and we define the family
\begin{equation}
 \widetilde{ A^{ \varepsilon}}: \mathcal R \to \cfsa(V) \textrm{ as } 
 \widetilde{ A^{ \varepsilon}}(t,s)\=  A_t -
 s\varepsilon\Id
\end{equation}
where $\mathcal R\=[a,b]\times [0,1]$. Being the rectangle topologically 
trivial  and by  invoking 
the homotopy property of the spectral flow it follows that the spectral flow of 
the path 
obtained by restricting $\widetilde{ A^\varepsilon}$ to the boundary of 
$\mathcal R$ 
is zero.
Thus by the path additivity property of the spectral flow, we have
\begin{equation}\label{eq:laravaelafava}
 \spfl( A_t(0); t \in [a,b]) =\spfl( A^{\varepsilon}(t); t \in [a,b])
 +  \spfl( A_a(s); s \in [0,1])- \spfl( A_b(s); s \in [0,1]).
\end{equation}
Since  by  the choice of $\varepsilon$  the paths $s\mapsto  A_a(s)$ and 
$s\mapsto  A_b(s)$ have no crossing instants other than, possibly, the initial 
and the final ones, we get 
\begin{equation}\label{eq:laravaelafava2}
 \spfl( A_t; t \in [a,b]) =\spfl( A^{\varepsilon}(t); t \in [a,b])
 -  \dim \ker  A_a+ \dim \ker  A_b
\end{equation}
This conclude the proof. 
\end{proof}
As direct consequence of Lemma \ref{thm:differnza-spfl}, we are entitled to 
give the 
following definition.
\begin{defn}\label{def:epsilon-sf}
We term {\em $\varepsilon$-spectral flow\/} of the path $ A$ and we 
denote it by  
$\spfl_\varepsilon( A_t;t\in[a,b])$, the spectral flow (as given in definition 
\ref{def:spectral-flow-unb}) of the 
path $ A^{\varepsilon}:[a,b] \to\cfsa(V)$; i.e.
\[
 \spfl_\varepsilon( A_t;t\in [a,b])\= \spfl( A^\varepsilon_t;t\in [a,b]).
\]
\end{defn}
By Definition \ref{def:epsilon-sf} and Lemma \ref{thm:differnza-spfl}, we get 
the following
\begin{equation}\label{eq:laravaelafava-giusta}
 \spfl( A_t;t\in [a,b]) =\spfl_\varepsilon( A_t;t\in [a,b]) -
 \dim \ker  A_a + \dim \ker  A_b.
\end{equation}
In particular if the endpoints of $ A$ are invertible, they coincide. 
We close this Section by recalling the relation between the spectral flow and 
another integer known in literature as  {\em relative Morse index\/}.
\begin{defn}\label{def:rel-morse-index}(\cite[Definition 2.8]{ZL99}).
 Let $ A , B\in \cfsa(V)$ and we assume that $ B$ is 
 $ A$-compact (namely compact in the graph norm topology of $ A$). Then the 
 {\em relative Morse index of the pair $ A$, $ A+ B$\/} is defined by 
 \[
  \irel( A,  A+ B)=-\spfl(\widetilde{ A};[a,b])
 \]
for $\widetilde{ A}\= A+ \widetilde{ B}$ and where $  \widetilde{ B}$ 
is any continuous curve of $ A$-compact operators such that $\widetilde{ B}(a)= 
0$ is the null operator and $ \widetilde{ B}(b)=  B$.
\end{defn}
\begin{rem}
 We observe that in the aforementioned paper the authors considered the more 
general case of bounded 
 Fredholm  operators in Banach spaces. However  
 the extension to the case considered above presents no difficulties. 
\end{rem}

% % % % % % % % % % % % % % % % % % % % % % % % % % % % % % % % % % % % % % % % 
% % 
% % 
% % 
% % 
% % % % % % % % % % % % % % % % % % % % % % % % % % % % % % % % % % % % % % % % 
% % 

\section{A spectral flow formula for families of Hamiltonian 
systems}\label{sec:spectral-flow-formula}

The scope of this Section which is the core of the paper  is twofold: on  one 
hand we 
construct an index theory and we prove a new  spectral flow formula for a 
one-parameter family of (linear) Hamiltonian systems defined  on unbounded 
intervals; on the 
other hand we recover the well-known index theory in the case of bounded 
intervals. 

We start by associating to the family of  Hamiltonian systems  a
\begin{itemize}
\item {\em geometrical index\/} defined in terms of the $\iCLM$-index 
of a suitable pair of Lagrangian paths defined in terms of the invariant 
(stable and unstable) subspaces and 
encoding the symplectic properties of the solution space;
\item {\em spectral  index\/}  will be given in terms 
of the spectral flow  of a path of first order elliptic operators and it is 
devoted to 
detect  the analytic properties of the arising path of closed selfadjoint 
Fredholm operators. 
\end{itemize}

 We start by recalling that  $T \in \Mat(2n,\R)$ 
is termed {\em hyperbolic\/} if its spectrum does not meet the imaginary axis. 
In this case, 
the spectrum of a hyperbolic operator $T$ consists of two isolated closed 
components (one of 
which may be empty) 
\begin{equation}\label{eq:convergence-2}
 \sigma(T) \cap \Set{z \in \C| \Re(z)<0} \textrm{ and } \sigma(T) \cap \Set{z 
\in \C| \Re(z)>0}.
\end{equation}
Let $\R^{2n}= V^-(T) \oplus V^+(T)$ be the corresponding $T$-invariant 
splitting of 
$\R^{2n}$ into closed 
subspaces, given by the spectral decomposition with projections $P^-(T)$ and 
$P^+(T)$. So 
\begin{equation}
 \begin{split}
  \sigma\left(T\vert_{V^-(T)}\right)= \sigma(T) \cap \Set{z \in \C| \Re(z)<0} &
\textrm{ and }\\
   \sigma\left(T\vert_{V^+(T)}\right)&= \sigma(T) \cap \Set{z \in \C| \Re(z)>0}.
 \end{split}
\end{equation}
Given  the one-parameter family of Hamiltonian systems 
\begin{equation}\label{eq:Ham-sys-bvp}
\begin{cases}
  z'(t)= S_\lambda(t)\, z(t) \qquad t \in \R\\
  \lim_{t \to -\infty} z(t)=0= \lim_{t \to +\infty} z(t)
 \end{cases}
\end{equation}
we define the two-parameter family of matrix-valued maps 
$\gamma_{(\tau,\lambda)}: \R \to \Mat(2n,\R)$ parametrised by  $(\tau,\lambda) 
\in  
 \R\times [0,1]$ as the fundamental solutions of the Hamiltonian systems given 
in Equation 
\eqref{eq:Ham-sys-bvp}, namely the matrix solutions of the following Cauchy 
problem
\begin{equation}\label{eq:cauchy}
 \begin{cases}
   \gamma'_{(\tau,\lambda)}(t)=S_\lambda(t)\gamma_{(\tau,\lambda)}(t), 
\qquad t \in \R\\
  \gamma_{(\tau,\lambda)}(\tau)=\Id
 \end{cases}
\end{equation}
where $S_\lambda(t)\=S(\lambda,t)$. We recall that the stable and unstable 
spaces of the Hamiltonian system given in Equation \eqref{eq:Ham-sys-bvp} are 
respectively 
given by 
\begin{equation}\label{eq:stabili-nonasym}
  E^s_\lambda(\tau)\=\Set{v \in \R^{2n}|\lim_{t\to +\infty} 
\gamma_{(\tau,\lambda)}(t)v=0} \textrm{ and }
   E^u_\lambda(\tau)\=\Set{v \in \R^{2n}|\lim_{t\to -\infty} 
\gamma_{(\tau,\lambda)}(t)v=0}.
\end{equation}
Let us  define the {\em asymptotic Hamiltonian systems \/}
\begin{equation}\label{eq:Ham-sys-bvp-asym}
\begin{cases}
  z'(t)= S_\lambda(\pm\infty)\, z(t) \qquad t \in \R\\
  \lim_{t \to -\infty} z(t)=0= \lim_{t \to +\infty} z(t)
 \end{cases}
\end{equation}
and as before we define the 
stable and unstable spaces as follows
\begin{multline}\label{eq:stabili-asym}
% % \begin{split}
 E^s_\lambda(\pm \infty)\=\Set{v \in \R^{2n}|\lim_{t\to +\infty} \exp{\big(t 
S_\lambda(\pm\infty)\big)}v=0} 
 \textrm{ and }\\
   E^u_\lambda(\pm \infty)\=\Set{v \in \R^{2n}|\lim_{t\to -\infty}\exp{\big(t 
S_\lambda(\pm\infty)\big)}v=0}.
% % \end{split}
\end{multline}
By invoking \cite[Proposition 1.2]{AM03}, the assumption (H1), implies the 
following  convergence result on the invariant manifolds
\begin{equation}\label{eq:conv-abbomendola}
 \lim_{\tau \to +\infty}E^s_\lambda(\tau)=E^s_\lambda(+\infty) 
\textrm{ and }  
 \lim_{\tau \to -\infty} E^u_\lambda(\tau)= E^u_\lambda(-\infty)
 \end{equation}
 where the limit is uniform with respect to the parameter $\lambda$ in the gap 
(metric) topology 
 of the  Grassmannian manifold.  
\begin{lem}\label{thm:stable-lagrangians}
If (H1) holds then the  space $E^s_\lambda(\tau)$, $(\tau, \lambda) \in 
[0,+\infty]\times [0,1]$,  
and  $E^u_\lambda(\tau)$,  $(\tau, \lambda) \in [-\infty,0]\times [0,1]$
defined  in Equations \eqref{eq:stabili-nonasym}-\eqref{eq:stabili-asym} belong 
to  $\Lagr(n)$.
\end{lem}
\begin{proof}
If $v,w:\R \to \R^{2n}$ are non-trivial solutions of the differential equation 
$z'- S_\lambda(t) z=0$ then 
$\omega\big(v(\tau), w(\tau)\big)=0$ is constant for all $\tau \in \R$. 
Moreover 
$\omega\big(v(\tau), w(\tau)\big)=0$ if $v(\tau), w(\tau) \in 
E^u_\lambda(\tau)$ or  
$v(\tau), w(\tau) \in E^s_\lambda(\tau)$, for some $\tau \in \R$. 
Thus $E^u_\lambda(\tau)$ and $E^s_\lambda(\tau)$ 
are isotropic subspaces of $(\R^{2n}, \omega)$. By taking into account   
(H1), $S_\lambda(\pm\infty)$ are hyperbolic matrices; thus in particular 
$E^s_\lambda(+\infty)=V^-\big(S_\lambda(+\infty)\big)$ 
and $E^u_\lambda(-\infty)=V^+ \big(S_\lambda(-\infty)\big)$ are Lagrangian 
subspaces. 
By  the convergence result stated in Formula
\eqref{eq:conv-abbomendola} and by taking into account that the dimension is 
a continuous integer-valued function  (and hence constant on the connected 
components of the Grassmannian), we conclude the proof. 
\end{proof}
\begin{lem}\label{thm:geo-well-defined}
 For every $\tau \in \R$, we have
 \[
  \iCLM\big( E^s_\lambda(0),E^u_\lambda(0)\big)=  \iCLM\big( 
E^s_\lambda(\tau),E^u_\lambda(\tau)\big).
 \]
\end{lem}
\begin{proof}
The proof relies on a very straightforward stratum homotopy argument. For, let 
us consider the following 
homotopy.
\[
 h\=(h_1, h_2): [0,1] \times [0,1] \to \Lagr(n) \times \Lagr(n), \qquad 
h(\lambda, 
s)= \big(E^s_\lambda((1-s)\cdot \tau),
 E^u_\lambda((1-s)\cdot \tau)\big).
\]
Since $E^u_\lambda(\tau')=\gamma_{(\tau,\lambda)}(\tau')E^u_\lambda(\tau)$ and 
$E^s_\lambda(\tau')=\gamma_{(\tau,\lambda)}(\tau')E^s_\lambda(\tau)$,
for every $\tau \in \R$, $\dim \big(E^u_\lambda(\tau) \cap 
E^s_\lambda(\tau)\big)$ is independent on $\tau$. By this 
argument  and by 
the stratum homotopy invariance 
property of the $\iCLM$ the thesis follows. This conclude the proof.  
\end{proof}\br
Let  $\mathcal R^-\=[-\infty,0]\times I$, $\mathcal R^+\=
[0,+\infty]\times [0,1]$ and let us define the 
continuous two-parameter family of Lagrangian subspaces: 
\begin{multline}
 E^u: \mathcal R^- \to \Lagr(n)\textrm{ defined by } E^u(\tau, 
\lambda)\=E^u_\lambda(\tau)
 \textrm{ and } \\
 E^s: \mathcal R^+ \to \Lagr(n) \textrm{ defined by } E^s(\tau, 
\lambda)\=E^s_\lambda(\tau).
\end{multline}\br
Being $\mathcal R^\pm$ topologically trivial, it follows that the pair 
$\big(E^s(0),E^u(0)\big) \in \P([0,1]; \R^{2n})$ pointwise given by   
$\big( E^s_\lambda(0), E^u_\lambda(0)\big)$ is stratum homotopic with respect 
to the endpoints to the 
pair 
$\big(E^s_0(\tau)*E^s_\lambda(+\infty)*E^s_1(-\tau),
E^u_0(-\tau)*E^u_\lambda(-\infty)*E^u_1(\tau)\big)$. 
\begin{prop}\label{thm:additivity}
The following equality holds:
\begin{multline}
 \iCLM(E^s_\lambda(0), E^u_\lambda(0);\lambda \in[0,1])= \iCLM\big(
E^s_0(\tau),E^u_0(-\tau); \tau \in[0,+\infty)\big) \\+ 
 \iCLM\big(E^s_\lambda(+\infty),E^u_\lambda(-\infty); \lambda \in [0,1]\big)- 
\iCLM 
\big(E^s_1(\tau),E^u_1(-\tau); \tau \in [0,+\infty)\big). 
 \end{multline}
\end{prop}
\begin{proof}
By the invariance of the $\iCLM$ index for stratum-homotopy and the 
additivity for concatenation of 
paths (cf. \cite[Properties II \& III, pag.127]{CLM94}, we have that 
\begin{multline}\label{eq:1}
\iCLM\big( E^s_\lambda(0), E^u_\lambda(0);\lambda \in [0,1]\big)\\= 
\iCLM\big(E^s_0(\tau)*E^s_\lambda(+\infty)*E^s_1(-\tau),
E^u_0(-\tau)*E^u_\lambda(-\infty)*E^u_1(\tau); \tau \in [0,+\infty)\cup 
[0,1]\cup(-\infty,0]\big)\\
 =\iCLM\big(E^s_0(\tau),E^u_0(-\tau);\tau \in [0,+\infty)\big)+ 
\iCLM\big(E^s_\lambda(+\infty),E^u_\lambda(-\infty);\lambda \in [0,1]\big)\\+
\iCLM\big(E^s_1(-\tau),E^u_1(\tau);\tau \in (-\infty,0]\big).
 \end{multline}
(Cf. Figure \ref{fig:stable}).  
By the reversal property of the $\iCLM$-index given in Section 
\ref{sec:spectral-flow-and-Maslov} 
(cf. \cite[Property X, pag. 130]{CLM94}),  we get that the last term in 
Equation 
\eqref{eq:1} can be written according to the following expression 
\begin{equation}\label{eq:2}
 \iCLM\big(E^s_1(-\tau),E^u_1(\tau);\tau \in (-\infty,0]\big)=- 
\iCLM\big(E^s_1(\tau),E^u_1(-\tau);\tau \in [0,+\infty)\big).
\end{equation}
The conclusion is obtained by putting together Equation \eqref{eq:1} and 
Equation \eqref{eq:2}. 
\end{proof}
Let $\mathcal W\=W^{1,2}(\R, \R^{2n})$ be the Sobolev space of all 
functions in $L^2(\R, \R^{2n})=:\mathcal V$ 
having derivatives in $\mathcal V$. By standard regularity arguments 
it readily  follows that any solution of the 
boundary value problems given in Equation \eqref{eq:Ham-sys-bvp} belongs to 
$\mathcal W$. Now, for each $\lambda \in [0,1]$, we define  the operators
\begin{equation}\label{eq:operatori}
  A_\lambda\= -J \frac{d}{dt} -B_\lambda(t): \mathcal W \subset 
 \mathcal V \longrightarrow \mathcal V.
\end{equation}
By invoking \cite[Theorem 2.1]{RS95}, it follows that for each $\lambda \in 
[0,1]$, 
$ A_\lambda$  is a closed unbounded 
selfadjoint Fredholm operator in $\mathcal V$ having domain $\mathcal  W$ and hence it 
remains a well-defined continuous path of closed selfadjoint Fredholm 
operators
\begin{equation}\label{eq:path-fredholm}
 A: [0,1] \longrightarrow \cfsa(\mathcal V): \lambda \longmapsto A_\lambda.
\end{equation}
\begin{defn}\label{def:spectral-index}
We term {\em spectral index\/} of the family of Hamiltonian systems given in 
Equation \eqref{eq:Ham-sys-bvp}, the integer $\ispec( A)$ defined as the  
spectral flow of the  path given in \eqref{eq:path-fredholm}; i.e.
\[
 \ispec(A)\= -\spfl(A; [0,1]).
\]
\end{defn}
For  $t_1, t_2 \in \R$ with $t_1 <t_2$ be fixed and let us denote by 
$\mathcal W_{\lambda,[t_1,t_2]}$ the 
Sobolev space defined by 
\[
\mathcal W_{\lambda,[t_1,t_2]}\=\Set{u \in W^{1,2}([t_1, t_2]; 
\R^{2n})| u(t_1) \in E^u_\lambda(t_1) 
 \textrm{ and } u(t_2) \in  E^s_\lambda (t_2)}.
 \]
Let us define the operators $ A_{\lambda,[t_1,t_2]}$  as  the restriction to $
W_{\lambda}([t_1,t_2];\R^{2n})$ of the operator $ A_\lambda$. 
\begin{prop}\label{thm:stessi-spfl}
For any  $t_1 ,t_2 \in \R$ with  $t_1 <t_2$ it follows that 
\begin{enumerate}
\item for every $\lambda \in [0,1]$, $A_{\lambda,[t_1,t_2]}$ is degenerate if 
and only if 
$A_\lambda$ is degenerate; i.e.
\[
 \dim \ker A_{\lambda,[t_1,t_2]}= \dim \ker A_\lambda.
 \]
\item $
 \spfl(A_{\lambda,[t_1,t_2]};\lambda \in[0,1])= \spfl(A_\lambda  ;\lambda \in 
[0,1])$.
\end{enumerate}
\end{prop}
\begin{proof}
($\Leftarrow$) We start to prove  that if $A_\lambda$ is degenerate 
then  $A_{\lambda,[t_1,t_2]}$ is degenerate. For, 
we assume that $ \dim \ker A_\lambda\neq\{0\}$ and we recall 
that 
 $ \dim \ker A_\lambda= \dim\big(E^u_\lambda(\tau)\cap E^s_\lambda(\tau)\big)$ 
for some 
 and hence for any $\tau \in \R$. Let $0 \neq v \in E^u_\lambda(\tau)\cap 
E^s_\lambda(\tau)$. 
Let  $v \in E^u_\lambda(\tau)\cap E^s_\lambda(\tau)$ then in particular $v \in  
E^u_\lambda(t_1)$ and 
$v \in  E^s_\lambda(t_2)$. Thus $v \in  E^u_\lambda(t_1)\cap E^s_\lambda(t_2)$ 
and hence $ v \in \ker 
 A_{\lambda, [t_1,t_2]}$ and this conclude the proof of the first part. \\
 ($\Rightarrow$) In order to prove this second implication it is enough to show 
that if
  $A_{\lambda,[t_1,t_2]}$ is non-degenerate then $A_\lambda$ is non-degenerate. 
  
 For any $t_1 <t_2$, we start by observing that 
$\gamma_{t_1,\lambda}(t_2)E^u_\lambda(t_1)=E^u_\lambda(t_2)$. 
Thus there exists $0 \neq v(t_1) \in E^u_\lambda(t_1)$. We now define 
$v(t)\=\gamma_{t_1,\lambda}\big(v(t_1)\big)$ and we observe that $ v(t_2) \in 
E^s_\lambda(t_2)$. 
In conclusion the non trivial function 
\[
v(t)\=\begin{cases}
 \gamma_{t_1,\lambda}(-t)\big(v(t_1)\big) & t \leq t_1\\
 \gamma_{t_1,\lambda}(t)\big(v(t_2)\big) & t_1 \leq t \leq t_2\\
 \gamma_{t_1+t_2,\lambda}(t)\big(v(t_1)\big) & t \geq t_2
\end{cases}
\]
belongs to  $\ker A_\lambda$. This conclude the proof of the first statement.

The proof of the second statement relies on the very definition of the spectral 
flow (cf. Definition \ref{def:spectral-flow-unb}). 
We start to  choose  a 
sufficiently small partition of the interval $[0,1]$, namely $0\=\lambda_0 < 
\dots < \lambda_n\=1$. Thus, we can find operators $ A_i \in \cfsa(\mathcal V)$ 
and positive real numbers $a_i$, $i=1, \dots, n$ in such a way the dimension of 
the 
spectral spaces $E_{[-a_i,a_i]}( A_\lambda)$ is constant for $\lambda \in 
[\lambda_{i-1}, \lambda_i]$. Up to refine the partition we can also assume that 
also the dimension of the spectral spaces $E_{[-a_i,a_i]}( A_{\lambda, 
[t_1,t_2]})$ 
is constant for $\lambda \in [\lambda_{i-1}, \lambda_i]$. 

Moreover we assume $E_{[-3a_i,3a_i]}( A_{\lambda, 
[t_1,t_2]})=E_{[-a_i,a_i]}( A_{\lambda, 
[t_1,t_2]})$ for $\lambda\in[\lambda_{i-1}, \lambda_i]$, that is, 
no eigenvalues belongs to the interval $(-3a_i,-a_i)\cup (a_i,3a_i)$. 
Let $\chi_{[t_1,t_2]}$ be the characteristic function of $[t_1,t_2]$, and we 
denote
$A_{\lambda, s,
[t_1,t_2]}:=A_{\lambda, [t_1,t_2]}+s \chi_{[t_1,t_2]} I$ for $s\in[0,2a_i]$.
By the homotopy property of spectral flow, from the fact that 
$\spfl(A_{\lambda, 
2a_i, [t_1,t_2]}; [\lambda_{i-1}, \lambda_i]))=0 $ we have 
\[
\spfl(A_{\lambda, [t_1,t_2]}; \lambda \in[\lambda_{i-1}, \lambda_i])) 
=\spfl(A_{\lambda_{i-1}, s, [t_1,t_2]}; s \in [0, 
2\,a_i]))-\spfl(A_{\lambda_{i}, s, 
[t_1,t_2]}; s \in [0, 2\,a_i]). 
\]
On the other hand, we let $A_{\lambda,s}:=A_{\lambda}+s \chi_{[t_1,t_2]} I$ for 
$s\in[0,2\,a_i]$. Similarly, we have 
\[
\spfl(A_{\lambda}; \lambda \in [\lambda_{i-1}, \lambda_i])) 
=\spfl(A_{\lambda_{i-1}, s}; 
s \in [0, 2\,a_i]))-\spfl(A_{\lambda_{i}, s}; s \in [0, 2\,a_i]). 
\]
We also observe that, for $\lambda\in[\lambda_{i-1},\lambda_i]$, 
\begin{multline}
\spfl(A_{\lambda, s, [t_1,t_2]}; s\in [0, 2\,a_i]))=\sum_{s=0}^{ 
2\,a_i}\dim\ker(A_{\lambda, s, 
[t_1,t_2]}) \\
 \spfl(A_{\lambda, s}; s \in [0, 
2\,a_i]))=\sum_{s=0}^{2\,a_i}\dim\ker(A_{\lambda, 
s}) .
\end{multline}
By arguing precisely as before,  we can conclude  
that 
\[
\dim\ker(A_{\lambda, s, [t_1,t_2]}) =\dim\ker(A_{\lambda, s})
\]
and by this last equality immediately follows that 
\[
  \spfl(A_\lambda; \lambda \in [\lambda_{i-1}, \lambda_i])= \spfl( A_{\lambda, 
[t_1,t_2]};\lambda \in 
  [\lambda_{i-1}, 
\lambda_i]).
\]
Summing up all over $i=0, \dots, n$, we get the thesis. This conclude the 
proof. 
\end{proof}

For each $\lambda \in [0,1]$ and for some $t_0 \in \R$, we set
\[
\mathcal X_{\lambda}\=\Set{u \in W^{1,2}([0,1], \R^{2n})| u(0) \in E^u_\lambda(t_0) 
\textrm{ and } 
u(1) \in E^s_\lambda(t_0)}
\]
and we consider the elliptic selfadjoint 
first order differential operator 
\[
 D_\lambda\=D\big(E^u_\lambda(t_0), E^s_\lambda(t_0)\big):
\mathcal X_{\lambda}\subset L^2([0,1]; \R^{2n}) \to L^2([0,1]; \R^{2n}) \textrm{ given 
by  }\]
\begin{equation}\label{eq:elliptic-operators}
  D_\lambda\=- J
 \dfrac{d}{dt}.
\end{equation}
\begin{prop}\label{thm:prop-prima-uguaglianza}
Let $ D$ be the  path pointwise defined by Equation 
\eqref{eq:elliptic-operators}.
The following equality holds
 \begin{equation}\label{eq:spflD=spflA}
  \spfl\big( D_\lambda ;\lambda \in [0,1]\big)= \spfl(A_\lambda ;\lambda \in 
[0,1]).
 \end{equation}
\end{prop}
\begin{proof}
Changing  variable by setting $ s\=(t-t_1)/(t_2-t_1)\in [0,1]$ the operator  
$A_{\lambda,[t_1,t_2]}$ 
can be re-written as follows
\begin{equation}\label{eq:ilprimo}
A_{\lambda,[t_1,t_2]}= -\dfrac{1}{t_2-t_1}J\dfrac{d}{ds} - 
B_\lambda\big((t_2-t_1)s+ t_1\big) :W_{\lambda}
 \subset L^2([0,1];\R^{2n}) \to
  L^2([0,1];\R^{2n})
\end{equation}
where 
\begin{equation}\label{eq:wlambda}
\mathcal  W_{\lambda}\=\Set{u \in W^{1,2}([0,1]; 
\R^{2n})| u(0) \in E^u_\lambda(t_1) 
 \textrm{ and } u(1) \in  E^s_\lambda (t_2)}.
\end{equation}
We now define the operator
\begin{equation}\label{eq:2altri}
\widetilde A_{\lambda,[t_1,t_2]}\=(t_2-t_1) A_{\lambda,[t_1,t_2]} =
-J\dfrac{d}{ds} - (t_2-t_1)B_\lambda\big((t_2-t_1)s+ t_1\big)
\end{equation}
on the domain $\mathcal W_\lambda$ defined in 
Equation \eqref{eq:wlambda} 
Let $\tau :[0,1]\to \R$  be the (continuous) map defined by $\tau(\sigma)\= 
t_2+(1-\sigma)(t_1-t_2)$ and 
let $k: [0,1]\times [0,1] \to \cfsa(L^2([0,1];\R^{2n}))$ given by 
\[
 k(\lambda, \sigma)\= - J\dfrac{d}{ds}-\big(t_2-\tau(\sigma)\big)
 B_\lambda\big(\tau(\sigma)\big).
\]
We observe that $k(\lambda, 0)= \widetilde A_{\lambda,[t_1,t_2]}$ 
and $k(\lambda, 1)= D_\lambda$.  
Being $t_2-t_1 >0$ we get that
\[
 \spfl(\tilde A_{\lambda, [t_1,t_2]};\lambda \in [0,1])=\spfl(A_{\lambda, 
[t_1,t_2]};\lambda \in [0,1]) 
\]
and by using the second statement of  
Proposition \ref{thm:stessi-spfl} we get that $\spfl(\tilde 
A_{\lambda \in [t_1,t_2]};\lambda \in [0,1])= 
 \spfl(A_\lambda ;\lambda \in [0,1])$. Passing to the limit for $(t_2 -t_1)\to 0$ 
in Equation 
\eqref{eq:2altri}, the 
 path $\widetilde A_{[t_1,t_2]}$ pointwise reduces to $D$. By the first 
statement of  
Proposition \ref{thm:stessi-spfl} 
 immediately follows that the homotopy is admissible since the dimension of the 
kernel of the map $k$ is 
 independent on the homotopy parameter $\sigma$. By the stratum 
 homotopy invariance of the  spectral flow and with respect to the endpoints 
the 
thesis readily follows.  This conclude the proof. 
\end{proof}\br
As direct consequence of Proposition \ref{thm:prop-prima-uguaglianza},  we  get
\begin{prop}\label{thm:main-nils} If assumption (H1) holds, 
 then 
 \[
 \iCLM\big( E^s_\lambda(0), E^u_\lambda(0);\lambda \in [0,1]\big)= 
-\spfl(A_\lambda ;\lambda \in [0,1]).
 \]
\end{prop}
\begin{proof} From \cite[Theorem 0.4]{CLM94}, we have 
\[ 
\spfl_\varepsilon( D_\lambda;\lambda \in [0,1])=\iCLM\big( E^u_\lambda(0), 
E^s_\lambda(0);\lambda \in 
[0,1]\big), 
\]
which implies 
\[-\spfl( D_\lambda;\lambda \in [0,1])=\iCLM\big( E^s_\lambda(0), 
E^u_\lambda(0);\lambda \in [0,1]\big).
\]
By invoking Proposition  
\ref{thm:prop-prima-uguaglianza} it is equal to 
$\spfl( A_\lambda ;\lambda \in [0,1])=\spfl( D_\lambda;\lambda \in [0,1])$.  
This conclude the proof. 
\end{proof}
\begin{rem} 
It is worth noticing that  Proposition \ref{thm:main-nils}  is  the  
generalization of the main result 
recently proved in by author in \cite{Wat15}.
\end{rem}
As consequence of Lemma \ref{thm:geo-well-defined} the integers 
$\dim\big(E^u_0(t_0)\cap E^s_0(t_0)\big)$ and $\dim\big(E^u_1(t_0)\cap 
E^s_1(t_0)\big)$ does not 
depend on $t_0$. Summing up all the results scattered so far we are in position 
to prove the  main result of this Section. 
\begin{thm}\label{thm:spectral-flow-formula}{\bf (Spectral flow formula)\/}
In the above notation and if  assumption (H1) holds, then we have
\begin{multline}
 \ispec(A)=\iCLM\big(E^s_0(\tau),E^u_0(-\tau);[0,+\infty)\big) + 
 \iCLM\big(E^s_\lambda(+\infty), E^u_\lambda(-\infty); [0,1]\big)\\ - \iCLM 
\big( E^s_1(\tau),E^u_1(-\tau); [0,+\infty)\big). 
\end{multline}
\end{thm}
\begin{proof}
By Definition \ref{def:spectral-index}, we know that $\ispec( A)= 
-\spfl( A;[0,1])$. The results is from  Proposition  
\ref{thm:main-nils} and  Proposition \ref{thm:additivity}. This conclude the 
proof. 
\end{proof}
As promised in Section \ref{sec:intro}, by using Theorem 
\ref{thm:spectral-flow-formula},
we are able to prove  a new spectral flow 
formula for Hamiltonian 
systems parametrised by bounded intervals.
Let $L, M \in \mathscr C^0\big([0,1];\Lagr(n)\big)$ and  let us consider 
the following family of Hamiltonian boundary value  problems
\begin{equation}\label{eq:Hamsysbounded}
\begin{cases}
   z'(t)= S_\lambda(t)\, z(t), \qquad t \in [a,b]\\
  \big(z(a), z(b)\big)\in L_\lambda \times M_\lambda.
 \end{cases}
\end{equation}
Following authors in 
\cite{Arn67, RS95, Lon02} and references therein, we associate to the 
Hamiltonian bvp given in Equation 
\eqref{eq:Hamsysbounded} the following Maslov index. 
Let $\gamma_{(\lambda,a)}(t)$ be denote  the two parameter family of 
matrix-valued maps 
$\gamma_{(\lambda,a)}: [a,b] \to \Sp(2n,\R)$  defined by
\begin{equation}\label{eq:cauchy-bounded}
 \begin{cases}
   \gamma'_{(\lambda,a)}(t)= S_\lambda(t)\,\gamma_{(\lambda,a)}(t), \qquad t 
\in [a,b]\\
  \gamma_{(\lambda,a)}(a)=\Id
 \end{cases}
\end{equation}
and we now consider the integer given by  $\iCLM(E_{\lambda, [a,b]}; \lambda \in 
[0,1])$,
where $E_{[a,b]}\in \P\big([0,1];\R^{2n}\big)$ is pointwise 
defined by $E_{[a,b]}(\lambda)\=\big( M_\lambda,\gamma_a(b)L_\lambda\big)$.
For  any $c \in (a,b)$, we consider the real-valued  functions on $[a,b]$ 
defined by
\begin{equation}\label{eq:alfa-beta}
 \alpha(t)\=
            \dfrac{(t-b)}{(a-b)}(c-b)+b  \ \ \textrm{ if }\ \  t \in [a,b]\\
           \quad
            \beta(t)\=
            \dfrac{(t-b)}{(a-b)}(c-a)+a  \ \ \textrm{ if } \ \ t \in [a,b].\\
\end{equation}
We observe that the function $\alpha$ is a positive affine reparameterisation 
of 
the interval $[c,b]$ with the same 
orientation whilst the function $\beta$ is  a negative affine 
reparameterisation 
of 
the interval  $[a,c]$ with the opposite orientation. For $c=(a+b)/2$ the 
functions 
introduced in Equation \eqref{eq:alfa-beta} reduce respectively to 
\[
  \alpha(t)=
            \dfrac12 (t-b) +b  \ \ \textrm{ if }\ \  t \in [a,b]\\
           \quad
            \beta(t)=
            \dfrac12(b-t) +a  \ \ \textrm{ if } \ \ t \in [a,b].
\]
Let $ c \in (a,b)$, $F\in \P([0,1];\R^{2n})$ be the continuous path of 
ordered pairs of Lagrangian paths pointwise defined by 
\[
 F(\lambda)\=\big( W_\lambda(c),V_\lambda(c)\big),
\]
where 
\begin{equation}
W_\lambda(c)\=\gamma^{-1}_a(b-c)M_\lambda.
\textrm{ and } V_\lambda(c)\=\gamma_{(a,\lambda)}(c)L_\lambda 
\end{equation}
\begin{lem}\label{thm:classical-maslov-comparison}
Under the previous notation, we have
\[
  \iCLM(F;[0,1])= \iCLM(E_{[a,b]};[0,1]).
\]
\end{lem}
\begin{proof}
The proof readily follows since
\begin{equation}
 \begin{split}
 &\iCLM\big(F;[0,1]\big)=   
\iCLM\big(\gamma^{-1}_a(b-c)M,
\gamma_a(c)L;[0,1]\big)\\
  &= 
\iCLM\big(M,\gamma_a(b-c)\gamma_a(c)L;[0,1]\big)= 
\iCLM\big(M,\gamma_a(b)L;[0,1]\big)=\iCLM(E_{[a,b]};[0,1]).
  \end{split}
  \end{equation}
This conclude the proof. 
\end{proof}
We denote by $ W^{1,2}([a,b];L_\lambda, M_\lambda)$ the Sobolev space defined 
by 
\[
 W^{1,2}([a,b],L_\lambda, M_\lambda)\=
 \Set{z \in W^{1,2}([a,b];\R^{2n})| z(a)\in L_\lambda \textrm{ 
and } z(b) \in  M_\lambda}
\]
and for each $\lambda \in [0,1]$, we define  the operators
\begin{equation}\label{eq:operatori-bounded}
 T_\lambda\= - J \frac{d}{dt} - B_\lambda(t):  W^{1,2}([a,b];L, M)
\subset L^2([a,b], \R^{2n})
 \longrightarrow L^2([a,b], \R^{2n}).
\end{equation}
It is well-known that for each $\lambda \in [0,1]$, 
$ T_\lambda$  is a closed unbounded 
selfadjoint Fredholm operator in $L^2([a,b], \R^{2n})$ having domain 
$  W^{1,2}([a,b];L_\lambda, M_\lambda)$  (cf. \cite{GGK90}, for instance) 
and hence it 
remains well-defined (gap continuous) path of closed selfadjoint Fredholm 
operators
\begin{equation}\label{eq:path-fredholm-2}
 T:[0,1]\ni \lambda \longmapsto T_\lambda \in  \cfsa(L^2([a,b],\R^{2n})).
\end{equation}

\begin{thm}\label{thm:main-2}\textbf{(Spectral flow formula on bounded 
intervals)}
In the previous notation, we have
\begin{multline}
 \ispec(T)=\iCLM\big(W_1\big(\beta(\tau)\big), 
V_1\big(\alpha(\tau)\big);[a,b]\big) + 
 \iCLM\big(M,L; [0,1]\big)\\ - \iCLM 
\big(W_0\big(\beta(\tau)\big), V_0\big(\alpha(\tau)\big); [a,b]\big) 
\end{multline}
\end{thm}
\begin{proof}
We define  the operators
\begin{equation}\label{eq:operatori-bounded1}
 D_\lambda\= - J \frac{d}{dt} :  W^{1,2}([a,b];V_\lambda(c), W_\lambda(c))
\subset L^2([a,b], \R^{2n})
 \longrightarrow L^2([a,b], \R^{2n}).
\end{equation}
By arguing as inthe proof of Proposition \ref{thm:prop-prima-uguaglianza}, we 
have 
$\spfl(D_\lambda ;\lambda \in [0,1])=\spfl(T_\lambda;\lambda \in [0,1])$. 
 We 
get also that 
\[
 \ispec(T)=-\iCLM(W_\lambda(c), V_\lambda(c);\lambda \in [0,1]).
\]
Arguing as in the proof of Theorem \ref{thm:spectral-flow-formula} and using  
the additivity properties of the $\iCLM$-index, we get the result. This conclude 
the proof. 
\end{proof}
\begin{cor}\label{thm:mainbounded}
 In the assumption of Theorem \ref{thm:main-2}, if $(L_\lambda, M_\lambda) 
\equiv (L,M) \in \Lagr(n)\times \Lagr(n)$, 
 then we have
\begin{equation}
 \ispec(T)=\iCLM\big(W_1\big(\beta(\tau)\big), 
V_1\big(\alpha(\tau)\big);\tau \in [a,b]\big)  - \iCLM 
\big(W_0\big(\beta(\tau)\big), V_0\big(\alpha(\tau)\big);\tau \in  [a,b]\big). 
\end{equation}
\end{cor}
Inspired by the classical Morse-type Index Theorem for periodic 
solution  of Hamiltonian system (cf. \cite{Lon02} and 
references therein) we now prove a  Morse-type index Theorem for unbounded 
motions of a Hamiltonian system. 

{\bf Proof of Theorem 1.\/}
The proof of this result  in the case of heteroclinic/homoclinic motions  
immediately follows by Theorem \ref{thm:spectral-flow-formula}, 
Definition \ref{def:geometrical-index} and the previous discussion.  The 
other two formulas in Theorem \ref{thm:main1-intro} on the future and past 
halfclinic orbits,
can be directly obtained by the previous one simply by 
setting for $i=0,1$ simply 
$E^u_i(-\tau)\equiv L_i$ in the case of future heteroclinic orbit and 
$E^s_i(\tau)\equiv L_i$ in the case of past heteroclinic orbit for any $\tau 
\in [0,+\infty)$. This 
conclude the proof.

{\bf Proof of Corollary \ref{thm:main2-intro} and Corollary 
\ref{thm:main3-final-intro}.\/}
Let us consider on the  space $\mathcal W$  
the  path of closed selfadjoint Fredholm operators pointwise respectively given 
by  
$A_\lambda\=-J\dfrac{d}{dt} - B_*-\lambda\big(B(t)-B_*\big)$. 
We observe that $\igeo(w_0) =0$ since $E^u_0(\tau),E^s_0(\tau)$ is constant. 
Similarly,   $E^s_\lambda(+\infty),E^u_\lambda(-\infty)$ not depend on $\lambda$ 
implies 
$ \iCLM\big(E^s_\lambda(+\infty),E^u_\lambda(-\infty); \lambda \in 
[0,1]\big)=0$.
This conclude the proof of Corollary \ref{thm:main2-intro}. 
The proof of  Corollary \ref{thm:main3-final-intro} is completely analogous. 

{\bf Proof of Theorem \ref{thm:main-2-intro}.\/} The proof of this result 
readily follows 
by Definition \ref{def:geometrical-index-bounded} and Theorem \ref{thm:main-2}. 
This conclude the proof.

% % % % % % % % % % % % % % % % % % % % % % % % % % % 

\vspace{0.5cm}
	\noindent
	\textsc{Prof. Xijun Hu}\\
	Department of Mathematics,
	Shandong University\\
	Jinan, Shandong, 250100 \\
	The People's Republic of China,
	China\\
	E-mail: \email{xjhu@sdu.edu.cn}

\vspace{0.5cm}
\noindent
\textsc{Prof. Alessandro Portaluri}\\
DISAFA, 
Università degli Studi di Torino\\
Largo Paolo Braccini 2 \\
10095 Grugliasco, Torino, 
Italy\\
Website: \url{aportaluri.wordpress.com}\\
E-mail: \email{alessandro.portaluri@unito.it}

\vspace{0.5cm}
\noindent
COMPAT-ERC Website: \url{https://compaterc.wordpress.com/}\\
COMPAT-ERC Webmaster \& Webdesigner: Arch.  Annalisa Piccolo

\end{document}